\documentclass[12pt]{amsart}
\usepackage{amssymb, amsmath, amsfonts}

\usepackage[raiselinks=false,colorlinks=true,citecolor=blue,urlcolor=blue,linkcolor=blue,bookmarksopen=true,pdftex]{hyperref}

\newtheorem{theorem}{Theorem}[section]
\newtheorem{proposition}[theorem]{Proposition}

\newtheorem{remark}[theorem]{Remark}
\newtheorem{definition}[theorem]{Definition}
\newtheorem{example}[theorem]{Example}
\newtheorem{corollary}[theorem]{Corollary}

\newcommand{\bc}{\mathbb{C}}
\newcommand{\br}{\mathbb{R}}
\newcommand{\bs}{\mathbb{S}}
\def\Span{\operatorname{span}}

\renewcommand{\hom}{\textrm{Hom}}
\newcommand{\wt}{\widetilde}
\newcommand{\rank}{\textrm{rank}}

\begin{document}
\baselineskip=15.5pt
\title{The vector field problem for homogeneous spaces}
\author[P. Sankaran]{Parameswaran Sankaran}
\address{The Institute of Mathematical Sciences (HBNI) \\
CIT Campus, Taramani, Chennai 600113.}
\email{sankaran@imsc.res.in}
\date{}
\dedicatory{Dedicated to Professor Peter Zvengrowski with admiration and respect.}
\subjclass[2010]{57R25}
\keywords{Vector fields, span, stable span, parallelizability, stable parallelizability, homogeneous spaces}
\thanks{This research was partially supported by the Department of Atomic Energy, through a 
XII Plan Project.}
\maketitle

\begin{abstract}
Let $M$ be a smooth connected 
manifold of dimension $n\geq 1$.  A vector field 
on $M$ is an association $p\to v(p)$ of a tangent 
vector $v(p)\in T_pM$ for each $p\in M$ which varies continuously with $p$. In more technical language it is a 
(continuous) cross-section of the tangent bundle $\tau(M)$.  The vector field problem asks: Given $M$, what is 
the largest possible number $r$ such that there exist 
vector fields $v_1,\ldots, v_r$ which are everywhere linearly independent, that is, $v_1(x),\ldots,v_r(x)\in T_xM$ are linearly independent for every $x\in M$.  The number $r$ is called the span of $M$, written $\Span(M)$.  It is clear that $0\leq \Span(M)\leq \dim (M)$.       
The vector field problem is an important and 
classical problem in differential topology.  
In this survey we shall consider the vector field problem focussing mainly on the class of compact homogeneous spaces.

\end{abstract}

\section{Introduction}
Let $M$ be a smooth connected manifold of dimension $n\ge 1$.  All manifolds we consider will be 
assumed to be paracompact and Hausdorff. If $p\in M$, the tangent space to $M$ at $p$ will be denoted 
$T_pM$.  We denote the tangent bundle of $M$ by $\tau M$ and its total space $\bigcup _{p\in M}T_pM$ 
by $TM$.  The projection of the bundle is denoted $\pi:TM\to M$; thus $\pi$ maps $T_pM$ to $p$.

A {\it vector field} $v$ on $M$ is an assignment $p\mapsto  v(p)\in T_pM$ of a tangent vector at $p$ for each $p\in M$ 
which varies continuously with $p$; thus $v:M\to TM$ is continuous and $\pi\circ v=id_M$.  In other words, $v$ is a 
continuous cross-section of the tangent bundle.

We are concerned with the following problem:

\noindent
{\bf The vector field problem:}  Let $M$ be a smooth 
manifold.  Determine the maximum number $r$ 
of everywhere linearly independent vector fields on $M$. 
Thus $r$ is the largest non-negative integer---denoted $\Span(M)$--- such that 
there exist (continuous) vector fields $v_1,\ldots, v_r$ on $M$ such that $v_1(p),\ldots , v_r(p)\in T_pM$ are linearly independent for every $p\in M$.  

It turns out that, in the vector field problem, if we require 
the vector fields to be smooth, then the resulting number 
$r$ is unaltered.  This is a consequence of the basic fact that the space of all 
smooth functions on a (smooth) manifold is dense in the space of 
all continuous functions.   
So, we may work with smooth vector fields throughout.  
Observe that $0\leq \Span(M)\leq \dim M$. 

In this largely expository article, we address the above problem for an important class of manifolds, namely,  
homogeneous spaces.  
After discussing some basic examples, we 
consider the problem for the spheres $\mathbb S^{n-1}$ whose solution at various stages 
brought along with it many new ideas and developments in algebraic topology.    Next we survey 
general results which are applicable to any compact connected smooth manifolds starting with 
Hopf's theorem, criterion for existence of a $2$-field, the result of Bredon and Kosi\'nski and of Thomas on the 
possible span of a stably parallelizable manifold, Koschorke's results on when span and stable span 
are equal, etc.  In \S 2, we consider the vector field problem for homogeneous spaces for a compact 
connected Lie group.  After elucidating the general results, mainly due to Singhof and Wemmer, 
for simply connected compact homogeneous spaces,  we consider certain special classes of 
homogeneous spaces (which are not necessarily simply connected) including projective Stiefel manifolds, 
Grassmann manifolds, flag manifolds, etc.  The only new result in this section is Theorem \ref{quotsofsun}, due to 
Sankaran.
In \S 3, we consider homogeneous spaces for non-compact Lie groups.  More precisely, we consider the class of 
solvmanifolds and compact locally symmetric spaces $\Gamma\backslash G/K$ where $G$ is a real semisimple 
linear Lie group without compact factors, $K$ a maximal compact subgroup of $G$ and $\Gamma$ a uniform lattice 
in $G$.   
Theorems \ref{torusbundle},  \ref{stronglypolygps} and \ref{locallysymmetric} are due to Sankaran (unpublished).

There are already at least two survey articles on the vector field problem. The paper by E. Thomas \cite{thomas2}, published in 1968, 
gives lower bounds for span in a general setup, whereas the main focus of the paper  
by J. Korba\v s and P. Zvengrowski \cite{kz}, published in 1994, was mostly on flag manifolds and projective Stiefel manifolds.  
See also \cite{kz96}, \cite[S4]{korbas}.   
While certain amount of overlap with these papers is unavoidable, 
the present survey emphasises the vector field problem 
for homogeneous spaces.  

 It seems that, in spite of much activity in this area, the determination 
of span of many families of homogeneous spaces (such as real Grassmann manifolds) remains a wide 
open problem. 
 I hope it would be useful to young researchers and new entrants to the field.

\subsection{First examples.}
We begin by giving some basic examples of vector fields on manifolds.  

If $M$ is an open subspace of $\mathbb{R}^n$ then $\Span (M)=\dim(M)=n$.  To see this, let $x_1,\ldots, x_n:M\to \mathbb{R}$ be the usual 
coordinate functions on $\mathbb{R}^n$ restricted to $M$.  Then $v_j(p):=\frac{\partial}{\partial x_j}|_p$, $1\le j\le n$ 
are linearly independent tangent vectors to $M$ at $p\in M$.   

\begin{example}\label{basicexamples}
{\em 
(i) Let $M=\bs^1$.  Then $\bs^1\ni (x,y)\mapsto -y\frac{\partial} {\partial x}+x\frac{\partial}{\partial y}\in T_{(x,y)}\mathbb{S}^1$ is a (smooth) tangent 
vector field on $\bs^1$.  So $\Span (\bs^1)=1=\dim \bs^1$.

(ii)  Consider the $n$-dimensional sphere $\mathbb{S}^{n}$ consisting of unit vectors in the Euclidean space 
$\mathbb{R}^{n+1}$.  We regard the tangent space to $\mathbb{S}^n$ at any point $x=(x_0,\cdots, x_n)$ as the vector 
subspace $\{x\}^\perp\subset \mathbb{R}^{n+1}$.  When $n=3$, we may regard $\mathbb{R}^4$ as the division algebra 
of quaternions over $\mathbb R$, generated by $i,j$ where $i^2=-1=j^2, k:=ij=-ji$.   The sphere $\mathbb{S}^3$ is the space 
of unit quaternions.  Multiplication (on the left) by 
the quaternion units $ i,j,k$ yields vector fields $v_1,v_2,v_3$ on $\mathbb S^3$: 
\[v_1(q)=(-q_1,q_0,-q_3,q_2)=iq, ~
v_2(q)=(-q_2, q_3, q_0,-q_1)=jq, ~
v_3(q)=(-q_3,-q_2, q_1,q_0)=kq.
\] for $q=q_0+q_1i+q_2j+q_3k=(q_0,q_1,q_2,q_3)\in \mathbb{S}^3$. 
 Then it is readily checked that $v_r(q)\perp q, r=1,2,3$, so that 
 $v_j$ are indeed vector fields on $\bs^3$.  Moreover, 
 $v_r(q)\perp v_s(q), r\neq s,$ and $||v_r(q)||=1$ for all $q\in \bs^3$.  
 Thus  $v_1, v_2, v_3$ are everywhere linearly independent vector fields on $\bs^3$ and we 
 conclude that 
 $\Span(\bs^3)=3=\dim \bs^3$. 

Using the multiplication in the 
octonions, one can write down explicitly seven everywhere linearly 
independent vector fields on $\bs^7$, as we shall now explain.   
The algebra of octonions, denoted $\mathbb{O}\cong \mathbb{R}^8$, was first discovered by 
Graves and shortly thereafter independently by Cayley and is also known as the 
Cayley algebra. 
The algebra $\mathbb{O}\cong \mathbb{R}^8$ is a non-commutative, 
non-associative division algebra generated over $\mathbb{R}$ by $e_i, 1\le i\le 7$, with multiplication 
defined by $e_ie_{i+1}=e_{i+3}, e_{i+1}e_{i+3}=e_i, e_{i+3}e_i=e_{i+1}, e_i^2=-1$, $e_ie_j=-e_je_i$ for $1\le i\ne j\le 7$ 
where the indices are read mod $7$.  Denote by $e_0$ the multiplicative 
identity $1\in\mathbb{R}\subset\mathbb{O}$.  Multiplication by $e_j$ preserves the Euclidean norm on $\mathbb{O}$ 
where the standard inner product is understood to be with respect to the basis $e_j, 0\le j\le 7$.  (To see this, 
we need only observe that left multiplication by $e_j$ permutes the basis elements up to a sign $\pm 1$.)

Now define $v_j:\mathbb{S}^7\to \mathbb{S}^7$ by $v_j(x)=e_j\cdot x, 0\le i\le 7.$  
Then $v_0(x)=x$, $||v_j(x)||=||x||=1$ and by straightforward verification $v_i(x)\perp v_j(x), 0\le i<j\le 7,$ for all $x\in \mathbb{O}$.  Thus $v_j, 1\le j\le 7,$ are vector fields on $\mathbb{S}^7$ which are everywhere linearly independent.  
Thus $\Span(\mathbb{S}^7)=7$.

(iii)  Suppose that $G$ is a Lie group and let $v\in T_eG$, where $e$ denotes the identity 
element.  Then we obtain a vector field, again denoted $v$ on $G$ by setting $v(g):=T\lambda_g(v)\in T_gG$ where $\lambda_g:G\to G$ is the left multiplication by $g$, sending $ x$ to $gx$. Note that $T\lambda_h(v(g))=T\lambda_h\circ 
T\lambda_g(v)=T(\lambda_h\circ\lambda_g)(v)=T\lambda_{hg}(v)=v(hg)$.  Thus $v$ is a {\it left-invariant 
vector field} on $G$. Conversely, every left-invariant vector field on $G$ is determined by its value at the identity.
Thus $T_eG$ is identified with the vector space of all left vector fields on $G$.   
If $v_1,\ldots, v_n$ form a basis for $T_eG$, then the left invariant vector fields $v_1,\ldots, v_n $ 
are everywhere linearly independent.  In particular $\Span(G)=\dim G$. 

The 
Lie bracket of two left invariant vector fields is again left invariant, making $T_eG$ a Lie algebra; it is the {\it Lie 
algebra} of $G$ and is denoted $\frak{g}$.

(iv) The above example can be generalized to principal $G$-bundles as we shall now explain.
Let $\pi: P\to M$ be the projection of a smooth principal bundle 
over a smooth manifold $M$ with fibre and structure group a Lie group $G$.  Let $v\in \mathfrak{g}$ and let $p\in P$. 
Identifying $G$ with the orbit $Gp\subset P$ through $p$, we obtain a tangent vector $\wt{v}_q\in T_qP$ that corresponds 
to $v_g$ where $g.p=q\in Gp\subset P$.  Since $v$ is a 
left invariant vector field, and since the $G$ action on $P$ corresponds to left multiplication in the Lie group $G$, $\wt{v}_q$ 
does not depend on the choice of $p$ and so yields a vector field $\wt v$ on $P$.  A choice of a basis $v_1,\cdots, v_n$ for $\mathfrak{g}$ yields everywhere 
linearly independent vector fields $\wt v_1,\ldots,\wt v_n$.  So we see that 
$\Span(P)\ge \dim(G)$.

(v)  Suppose that $p:M\to N$ is a covering projection where $M,N$ are smooth manifolds and $p$ is smooth. 
Let $\Gamma$ be the deck transformation group (which acts on the left of $M$).  Then action of $\Gamma$ on $M$ is via diffeomorphisms and so we have an induced action of $\Gamma$ on $TM$: if $\gamma\in \Gamma$ and 
$e\in T_xM$ then $\gamma(e):=T\gamma_x (e)\in T_{\gamma.x}M$.  We shall write $\gamma_*e$ instead of $T_x\gamma(e)$. 
Also if $x\in M$, 
then $Tp_x:T_xM\to T_{p(x)}N$ is an isomorphism of vector spaces.  
Since $p=p\circ \gamma$ for any $\gamma\in \Gamma$,  $Tp_x(e)=Tp_{\gamma.x}(\gamma_*e)$.
Therefore the tangent vector at $Tp_x(e)\in T_{p(x)}N$ may be identified with the set $
[x,e]:=\{(\gamma.x,\gamma_*(e))\mid \gamma\in \Gamma\}.$
If $v$ is a smooth vector field on $N$, then 
we can `lift' it to a smooth vector field $\wt{v}$ on $M$ defined by $p_*(\wt{v}(x))=v(p(x))$ where we have written $p_*$ for $Tp$. 
If $v_1,\ldots, v_k$ is a $k$-field on $N$, then so is $\wt v_1,\ldots, \wt v_k$.  Hence $\Span(M)\ge \Span(N)$.  

We observe 
that the vector field $\wt v$, obtained as a lift of a vector field $v$ on $N$, is $\Gamma$-invariant, that is, $\wt v_{\gamma(x)}=\gamma_*(\wt v(x))$ for all $\gamma\in \Gamma$.   Conversely, if $u$ is a $\Gamma$-invariant vector field on $M$ it is the lift of a unique vector field $v$ on $N$.  
}
\end{example}

\begin{definition}
We say that a manifold $M$ is parallelizable  if 
$\Span(M)=\dim (M)$.  
\end{definition}

We have seen already that any Lie group is parallelizable as also the spheres $\mathbb{S}^1,\mathbb{S}^3, \mathbb{S}^7$.  
Bott and Milnor \cite{bm}  
and independently Kervaire \cite{kervaire} showed that these are the only parallelizable spheres (besides $\mathbb{S}^0$). 

\subsection{Span of spheres and projective spaces} \label{spanofspheres}

Radon \cite{radon} and Hurwitz \cite{hurwitz} independently obtained the following algebraic result which yields 
a {\it lower bound} for the span of spheres.

A bilinear map $\mu: \mathbb R^k\times \mathbb R^n\to \mathbb R^n$ 
 is called an {\it orthogonal multiplication} if $||\mu(u,v)||=||u||.||v||$ for all $u\in \mathbb R^k, v\in \mathbb R^n$. 
 Given an orthogonal multiplication $\mu$ and an orthogonal transformation $\phi$ of $\mathbb R^n$ we 
 see that the bilinear map $\phi\circ \mu $ is again an orthogonal multiplication. 
Also if $u\in \mathbb R^k$ is a unit vector, then $\mu_u:\mathbb R^n\to \mathbb R^n$ defined as $v\mapsto \mu(u,v)$ is an 
orthogonal transformation.  Using these observations, one may 
normalize $\mu$ so that $\mu(e_1,y)=y,~ \forall y\in \mathbb R^n$.  
 The proposition below relates the existence of an orthogonal multiplication to the span of $\mathbb S^{n-1}$. 

\begin{proposition} \label{orthog} \label{radonhurwitz}
If there exists an orthogonal multiplication $\mu:\br^k\times \br^n\to \br^n$, 
then $\Span(\bs^{n-1})\geq k-1$.  
\end{proposition}
\begin{proof} Without loss of generality we assume that $\mu(e_1,y)=y~\forall y\in \mathbb R^n$.
Let $x\in \br^k$ be a unit vector and let $\mu_x(y)=\mu(x,y)$.   As observed already, $\mu_x$ is an orthogonal operator.  
Wring $\mu_i$ for $\mu_{e_i}, 1\leq i\leq k$, we claim that if $y\neq 0$, then  
$\mu_i(y), 1\leq i\leq k,$ are linearly independent.  
Let, if possible, $\sum a_i\mu_i(y)=0$ with some $a_i\neq 0$.  Multiplying by a scalar if necessary, 
we assume without loss of generality that $\sum_{1\leq i\leq k}a_i^2=1$ so that $a:=\sum_{1\leq i\leq k}a_ie_i$
is a unit vector.  Hence $0=\sum a_i\mu_i(y)=\mu_a(y)$ implies that $y=0$ as $\mu_a$ is an orthogonal transformation.  This establishes our claim.  

Next we claim that for any non-zero $v\in \mathbb{R}^{n}$, $\mu_i(v)\perp \mu_j(v)$ whenever $i\neq j$.  
Set $a:=e_i+e_j\in \br^k$.
Then, using bilinearity,  for 
$2||v||^2=||\mu(a,v)||^2=||\mu_i(v)+\mu_j(v)||^2=||\mu_i(v)||^2+||\mu_j(v)||^2+2\langle \mu_i(v),\mu_j(v)\rangle$.
Since $||\mu_i(v)||=||\mu_j(v)||=||v||$, we see that $\mu_i(v)\perp \mu_j(v)$.

Since $\mu_1=id$, we have shown that the $v\mapsto \mu_j(v)$ are 
vector fields on $\bs^{n-1}$ which are everywhere linearly independent.  Hence $\Span(\bs^{n-1})\geq k-1$. 
\end{proof}

When $\mu_1=id$, it is not difficult to show that 
the $\mu_i=\mu_{e_i}, 2\le i\le k,$ are skew symmetric orthogonal transformations of $\mathbb{R}^n$ so that $\mu_i^2=-id$, and moreover they  
satisfy the relations $\mu_i\mu_j=-\mu_j\mu_i, ~i\neq j, 2\leq i,j\leq k$.   
Conversely, 
if there exist skew symmetric orthogonal transformations $\mu_i$, $2\leq i\leq k,$ satisfying the above relations, then there 
exists an 
orthogonal multiplication $\mu:\br^k\times \br^n\to \br^n$ such that $\mu_i=\mu_{e_i}, i\ge 2,$ with $\mu_{e_1}=id$. 
The transformations $\mu_2,\ldots, \mu_k$ are known as the {\it Radon-Hurwitz 
transformations.}

A well-known and classical theorem of Hurwitz and Radon gives the maximum value of $k$ as in the above 
proposition for any given $n$.   Write $n=2^{4a+b}\times (2c+1)$ where $0\leq b\leq 3, a\geq 0, c\geq 0$ are integers. 
Then the maximum value of $k$ as in the above proposition is $k=\rho(n)$ where  
$\rho(n)=8a+2^b$, the {\it Radon-Hurwitz 
number} of $n$.  See also Eckmann \cite{eckmann}.   

\begin{theorem}  {\em (Radon \cite{radon}, Hurwitz \cite{hurwitz})}  Let $n\ge 2$ and let $\rho(n)$ denote the Radon-Hurwitz number defined above.  Then 
$\Span(\mathbb{S}^{n-1})\ge \rho(n)-1$.
\end{theorem}

We now state the celebrated theorem of Adams who showed that the Radon-Hurwitz lower bound 
is also the upper bound, thereby 
determining the span of the spheres. 

\begin{theorem}{\em  (Adams \cite{adams})} Let $n\ge 2$.  Then 
$\Span( \bs^{n-1})=\rho(n)-1$. \hfill $\Box$
\end{theorem}
The proof of this theorem uses $K$-theory and  Adams operations and is beyond the scope of these notes. 
The reader may refer to 
Husemoller's book \cite{husemoller} for a complete proof.

Note that with notations as in the proof of Proposition \ref{orthog}, the vector fields $\mu_j$ on the sphere $\mathbb{S}^{n-1}$ 
are {\it odd}, that is, $\mu_j(-v)=-\mu_j(v),~\forall v\in \mathbb S^{n-1}$.
Since $\mathbb{S}^{n-1}\to \mathbb{R}P^{n-1}$ is a covering projection 
with deck transformation group $\mathbb{Z}_2$ generated by the antipodal map, 
we see that the $\mu_j$ define vector fields $\bar{\mu}_j$ on the quotient 
space $\mathbb{S}^{n-1}/\mathbb{Z}_2=\mathbb{R}P^{n-1}$ the $(n-1)$-dimensional real projective space. 
(See Example \ref{basicexamples} (v).)  
Thus we have the lower bound $\Span(\mathbb{R}P^{n-1})
\ge \rho(n)-1$.  On the other hand, 
$\Span(\mathbb{S}^{n-1})\ge \Span(\mathbb{R}P^{n-1})$ again by the same Example.  
Hence Adams' theorem yields the following.

\begin{corollary}
$\Span(\mathbb{R}P^{n-1})=\rho(n)-1$. \hfill $\Box$
\end{corollary} 

We will see that the Radon-Hurwitz number arises as the lower bound for span of certain other 
homogeneous spaces as well.  

\subsection{Span and characteristic classes}  
The determination of the span of a manifold is in general a difficult 
problem.  However, techniques and tools of algebraic topology have been 
successfully applied to obtain invariants (or obstructions) whose vanishing (or non-vanishing) 
would lead to lower (or upper) bounds for the span.  It is generally the case 
that obtaining lower bound for span is much harder than finding invariants 
whose non-vanishing leads to upper bounds.  
The following result which gives a necessary and sufficient condition 
for span to be at least one is due to Hopf.

\begin{theorem}{\em (H. Hopf \cite{hopf})} \label{hopf}
Let $M$ be a compact connected smooth manifold. Then $\Span(M)\ge 1$ if and only if 
the Euler-Poincar\'e characteristic $\chi(M)$ of $M$ is zero.  \hfill $\Box$
\end{theorem}
We merely give an outline of the proof. 

First one shows that $M$ admits a smooth vector field $v$ which has only finitely many singularities---points 
where $v$ vanishes.  In fact, put a Riemannian metric on $M$.  Then 
$\textrm{grad}(f) $, the gradient vector field 
associated to a Morse function $f:M\to \mathbb R$, has only finitely many singularities. 
To each singular point $p\in M$ one associates an integer called the index of $v$ at $p$ and denoted 
$\textrm{index}_p(v)$ obtained as follows.  Choose a coordinate chart $(U,\phi)$ 
around $p$ such that $v_x\ne 0~\forall x\in U\setminus \{p\}$.  Take a small sphere $S\cong \mathbb S^{d-1}$ contained in $U$ 
centred at $p$ where $d=\dim M$.  Then $\phi$ induces an orientation on $U$ and hence on $S$ from the standard orientation 
on $\phi(U)\subset\mathbb R^d$.
The degree of the 
map $S\to \mathbb S^{d-1}$ defined as $x\mapsto v_x/||v_x||$ 
is defined to be the index of $v$ at $p$.  
Set $\textrm{index}(v):=\sum_p \textrm{index}_p(v)$ (where the sum is over the (finite) set of all  singular points of $v$); it is understood 
that if $v$ has no singularities, the $\textrm{index}(v)$ is zero.    It turns out that $\textrm{index}(v)$ is  
independent of the choice of the vector field $v$.

When $f$ is a Morse function on $M$, the singularities of $\textrm{grad}(f)$ 
are precisely the critical points of $f$ and, moreover, the index of $\textrm{grad}(f)$ at a critical point $p$ 
is either $+1$ or $-1$ depending on the parity of the index of the function $f$ at $p$.  (See \cite{milnor-mt}.)  Therefore we see that 
$\textrm{index(grad}(f))$ equals  $\sum_{0\le q\le d} (-1)^q c_q$ where $c_q$ is the number of 
$q$-dimensional cells in the CW structure on $M$ obtained from the Morse function $f$.  
As is well-known $\sum (-1)^q c_q=\chi(M)$.  

Denote by $\underline \pi_{d-1}\mathbb S^{d-1}=\underline{\mathbb Z}$ the local coefficient system associated to the 
unit tangent bundle $S(\tau M)\to M$.   If $M$ is orientable it is the constant coefficient system $\mathbb Z$; otherwise 
it is given by the homomorphism $\pi_1(M)\to Aut(\mathbb Z)\cong \mathbb Z_2$ with kernel the index $2$ subgroup 
corresponding to the orientation double cover of $M$.   In any case one has the Poincar\'e duality isomorphism 
$H^d(M;\underline {\mathbb Z})\cong H_0(M;\mathbb Z)=\mathbb Z$.  The obstruction to the existence of a cross-section 
of $S(\tau M)\to M$ is 
the {\it Euler class} $\mathfrak{e}(M)\in H^d(M;\underline {\mathbb Z})$.  The class $\mathfrak{e}(M)$ corresponds, under Poincar\'e 
duality, to the index of a vector field $v$ on $M$ with isolated singularities.  
Since $\textrm{index}(v)=\chi(M)$, vanishing of $\chi(M)$ implies the existence of a nowhere vanishing 
vector field.
We refer the reader to \cite{steenrod}, \cite{ms}, \cite{milnor-uvlect} for further details. 

When $\dim M=d$  is odd, the Euler-Poincar\'e characteristic of $M$ vanishes (by Poincar\'e duality) and so we have 

\begin{corollary} \label{hopf-odd-dim}
Suppose that $\dim (M)$ is odd. Then $\Span (M)\geq 1$. \hfill $\Box$
\end{corollary}

The notion of span extends in a natural way to any vector bundle over an arbitrary 
topological space.  The span of a real vector bundle $\xi$ over $X$ with projection $p:E(\xi)\to X$ is the maximum 
number $r$, denoted $\Span(\xi)$, such that there exist everywhere linearly independent cross sections 
$s_1,\ldots, s_r:X\to E(\xi)$.   The number $\rank(\xi)-\Span(\xi)$ is called the {\it geometric dimension} of $\xi$. 
Note that $0\le \Span(\xi)\le d$ where $d$ denotes the rank of $\xi$.   

If $X$ is a $d$-dimensional CW-complex (where $d$ is finite) or has the homotopy type of such a space, and if 
$\xi$ is a real vector bundle over $X$ such that $\rank(\xi)\ge d$, then $\Span(\xi)\ge \rank(\xi)-d;$  also 
the geometric dimension of 
$\xi\oplus k\epsilon$ is independent of the choice of $k\ge 1$.  See \cite{husemoller}.  Here, and in what follows,
$\epsilon$ denotes a trivial line bundle and $k\eta$ denotes the $k$-fold Whitney sum $\eta\oplus \cdots\oplus\eta$ 
of $\eta$ with itself.   
The notion of {\it stable span} of a manifold is defined as follows: 

\begin{definition}
Let $M$ be a connected smooth manifold.  The stable span of $M$, denoted $\Span^0(M),$ is defined to be the largest 
integer $r$ such that $\tau M\oplus k\epsilon =\eta\oplus (k+r)\epsilon$ where $k\ge 1$.  We say that $M$ is 
stably trivial if $\tau M\oplus k\epsilon$ is trivial for some $k\ge 1$.
\end{definition}

In view of the observation preceding the definition, we may take always $k=1$ to obtain $\Span^0(M)$.  In 
particular $M$ is stably trivial if and only if $\tau M\oplus \epsilon$ is trivial.  
Stably parallelizable manifolds are also known as $\pi$-manifolds. 

We shall assume familiarity with the definitions and properties of characteristic classes associated to 
vector bundles such as Stiefel-Whitney classes, Pontrjagin classes, etc. 
The standard reference for these is the book by Milnor and Stasheff \cite{ms}.   

The Stiefel-Whitney classes of a smooth manifold $M$, denoted $w_j(M)\in H^j(M;\mathbb Z_2)$,  
are by definition, the Stiefel-Whitney classes $w_j(\tau M)$ of the 
tangent bundle of $M$.  
Similar convention holds for Pontrjagin classes.  It turns out that the 
Stiefel-Whitney classes of a compact connected smooth manifold are independent of the {\it smoothness structure} and depend only 
on the underlying topological manifold.  This is because the total Stiefel-Whitney class $w(M)=\sum w_j(M)$ of $M$ can be 
described purely in terms of the cohomology algebra $H^*(M;\mathbb Z_2)$ and the action of the Steenrod algebra $\mathcal A_2$ 
on it.  So $w_j(M)$ are even homotopy invariants; see \cite[Ch. 11]{ms}.   In contrast, it is known that 
the Pontrjagin classes are not homotopy invariants.   (Note the Stiefel-Whitney classes are {\it not} homotopy invariants 
when the manifold is not compact.   For example, $w_1(M)\ne 0$ when $M$ is the M\"obius strip as it is not orientable whereas the 
cylinder $\mathbb S^1\times \mathbb R$ is parallelizable and so 
$w_1(\mathbb S^1\times \mathbb R)=0$.)

Recall that Stiefel-Whitney classes are `stable' classes: $w_j(\xi\oplus r\epsilon)=w_j(\xi)$ for all $j\ge 0, ~\forall r\ge 1,$ and that $w_k(\xi)=0$ if 
$k> \rank(\xi)$ for any vector 
bundle $\xi$.  It follows that,  
$\Span^0(M)\le r$ if $w_{d-r}(M)\ne 0$ for some $r\ge 0$.    Likewise, the Pontrjagin classes $p_j(\xi)$ are also stable 
classes, and, $p_j(\xi)=0$ if $j>\rank(\xi)$.  
So the non-vanishing of $p_{j}(M)$ implies that $\Span^0(M)\le \dim M-j$.

All spheres are stably parallelizable and so $\Span^0(\mathbb{S}^n)=\Span(\mathbb{S}^n)$ if and only 
if $n=1,3,7.$ On the other hand we have the following result, which is a special case of a more 
general result due to James and Thomas \cite{jt}. 

\begin{theorem} {\em (Cf. \cite[Corollary 1.10]{jt})}\label{proj-stablespan}
For any $n\ge 1$, $\Span^0(\mathbb{R}P^n)=\Span(\mathbb{R}P^n)$.
\end{theorem}
\begin{proof}
If $n$ is even, the Stiefel-Whitney class $w_n(\mathbb{R}P^n)\ne 0$ which shows that the stable span of $\mathbb{R}P^n$ vanishes. 
Also, trivially, the statement is valid if $n=1,3,7$. 

So assume that $n$ is odd and that $n\ne 1,3,7$.
James and Thomas \cite{jt} have shown that, for such an $n$, 
if $\eta$ is an $n$-plane bundle such that $\eta\oplus \epsilon\cong (n+1)\xi\cong \tau \mathbb{R}P^n\oplus \epsilon$, 
then $\eta$ is isomorphic to $\tau \mathbb{R}P^n$.   
This readily implies that the geometric dimension of $\tau \mathbb{R}P^n$ and of $\tau \mathbb{R}P^n\oplus \epsilon$ 
are equal---equivalently $\Span^0(\mathbb{R}P^m)=\Span(\mathbb{R}P^n).$
\end{proof}

\begin{definition}\label{semichar}
Let $M$ be a closed connected orientable manifold of dimension $n$ where $n=2m+1$ is odd.  The {\em  Kervaire mod $2$ semi-characteristic} of $M$ is defined as 
$\hat{\chi}_2(M):=\sum_{0\leq j\leq m} \dim_{\mathbb Z_2} H^{2j}(M;\mathbb{Z}_2)\mod 2.$  Likewise the {\em Kervaire real 
semi-characteristic} of $M$ is defined as 
$\kappa(M) =:\sum_{0\le j\le m} b_{2j}(M) \mod 2$ where $b_k(M)$ denotes the $k$-th Betti 
number of $M$.  
\end{definition}

Suppose that $M$ is not orientable but satisfies the weaker condition that $w_1(M)^2=0$.  Assume that $n\equiv 1\mod 4$.   Atiyah and 
Dupont  \cite{atiyahdupont} defined the {\it twisted Kervaire semi-characteristic}, denoted $R_L(M),$ using cohomology with coefficient in 
a local system $L$ of the field of complex numbers.   One may view $L$ as the complex line bundle  
associated to a covering projection $\wt{M}\to M$ with deck transformation group $\mathbb{Z}_4$.  (Thus the total space 
of $L$ is $\wt{M}\times_{\mathbb Z_4} \mathbb C$ where the action of $\mathbb Z_4$ on $\mathbb C$ is generated by 
multiplication by $i\in \mathbb S^1$.)    
Such a cover 
corresponds to a homomorphism $\pi_1(M)\to \mathbb{Z}_4$ or equivalently an element $u\in H^1(M;\mathbb{Z}_4)$.  The element $u$ is chosen so that $w_1(M)=u \mod 2$.  Such an element exists since $w_1(M)^2=0$.   
 The cohomology $H^*(M;L)$, which is the same as the de Rham cohomology with 
coefficients in $L$, admits a non-degenerate Poincar\'e pairing $H^{n-p}(M;L)\times H^p(M;L)\to H^n(M;\Omega^n\otimes \mathbb{C})\cong \mathbb{C}$ in view of the isomorphism $L\otimes L\cong \Omega^n\otimes \mathbb{C}$.  
(Here $\Omega^n$ is the determinant of the cotangent bundle of $M$.)
The twisted semi-characteristic is defined as 
$R_L(M)=(1/2)(\sum_{0\le k\le n}\dim_\mathbb{C}(H^k(M;L))) \mod 2$.   When $w_1(M)=0$, that is, when $M$ is orientable, 
then $L$ and $\Omega^n$ are trivial and we have $R_L=\kappa(M)$.     

We now state a result which gives necessary and sufficient conditions for the span to be 
at least $2$ (under mild restrictions on the manifold), similar in spirit to Hopf's theorem \ref{hopf}.
We refer the reader to \cite{kz} and \cite[\S 2]{thomas2} for a detailed discussion and relevant references. 

Recall that the signature $\sigma(M)$ of a compact connected oriented manifold of dimension $4m$ is the signature of 
the symmetric bilinear pairing $H^{2m}(M;\mathbb{R})\times H^{2m}(M;\mathbb{R})\to \mathbb{R}$ given by $(\alpha, \beta)\mapsto \langle \alpha\cup\beta,\mu_M\rangle$ where $\mu_M\in H_{4m}(M;\mathbb{Z})\hookrightarrow 
H_{4m}(M;\mathbb{R})\cong \mathbb{R}$ denotes the fundamental class of $M$. 

\begin{theorem} \label{span2} {\em (See \cite[\S2]{thomas2})}
Suppose that $M$ is a compact connected oriented smooth manifold of dimension $d\ge 5$.  Then 
$\Span(M)\ge 2$ if and only if one of the following holds (depending on the value of $d\mod 4$):\\
 (a) $d\equiv 1\mod 4$ and $w_{d-1}(M)=0, \kappa(M)=0$; \\(b) $d\equiv  2 \mod 4$ and $\chi(M)=0$; \\
 (c) $d\equiv 3\mod4$;\\
  (d) $d\equiv 0\mod 4, $ and $\chi(M)=0, \sigma(M)\equiv 0\mod 4$.  
\end{theorem}


We shall now explain the approach of Koschorke \cite{koschorke} who regarded a sequence of $r$ vector fields on $M$ as 
a vector bundle homomorphism $r\epsilon \to \tau M$ and constructed  
obstruction classes $\omega_{r}$ which live in the normal bordism group $\Omega_{r-1}(\mathbb RP^{r-1}\times M;\phi_M)$ 
for a suitable virtual vector bundle $\phi_M=\phi^+-\phi^-$ over $\mathbb RP^{r-1}\times M$. 

Let $r<n/2$ and let $\underline X=X_1,\ldots, X_r$ be a sequence of vector fields on a smooth manifold $M$.  
\footnote{Such a sequence is referred to as an $r$-field in \cite{koschorke}, but in the literature it is also often used  
to mean one which is everywhere linearly independent. So we avoid this terminology altogether.}
A point $p\in M$ is a {\it singularity} of $\underline X$ if $X_{1,p},\ldots, X_{r,p}\in T_pM$ is linearly dependent.  
Denote by $S=S(X)$ the singularity set, that is, 
$S:=\{p\in M\mid X_{1,p},\ldots, X_{r,p}~\textrm{is linearly dependent}\}$.  We say that $\underline X$ 
is {\it non-degenerate} if the following conditions hold: (a) $\forall p\in S$, the vectors
$X_{1,p},\ldots, X_{r,p}$ span a subspace of $T_pM$ of dimension $r-1$, (b)  $S$ is a compact smooth submanifold of $M$ 
of dimension $(r-1)$, 
(c) the map $M\ni p\mapsto (X_{1,p},\ldots, X_{r,p})\in E(r\tau M)$ is transverse to the (closed)
subspace $D_{r-1}:=\cup_{p\in M} D_p^{r-1}$ where $D_{p}^{r-1}=\{(u_1,\ldots,u_r)\mid u_j\in T_pM, 
1\le j\le r,~\textrm{span a linear space of dimension~}\le r-1\}$.  It turns out that when $2r<n$, there always exists a 
non-degenerate sequence $\underline{X}$. 
Note that $S$ meets $D_{r-1}$ along the {\it submanifold} $A_{r-1}=D_{r-1}\setminus D_{r-2}$ of $M$.    
Non-degeneracy guarantees a well-defined embedding $g:S\to \mathbb RP^{r-1}\times M$ obtained as 
$p\mapsto ([a_1,\ldots, a_r], p)$ where $\sum a_jX_{j,p}=0$ where not all $a_j$ are zero.   
Consider the virtual bundle $\phi_M:=\phi^+-\phi^-$ over $\mathbb RP^{r-1}\times M$ where $\phi^+=\xi\otimes \tau M, 
\phi^-=r\xi\oplus \tau M$. (Here $\xi$ denotes the Hopf bundle over the projective space $\mathbb RP^{r-1}$.) 
Then $g^*(\phi_M)$ is a stable normal bundle over $S$; more precisely, there is a vector bundle isomorphism 
$\bar{g}: s\epsilon\oplus \tau S\oplus g^*(\phi^+)\cong t\epsilon\oplus g^*(\phi^-)$ that covers $g$ where 
$\phi_M=\phi^+-\phi^-$.  This leads to a well-defined 
obstruction class $\omega_{r}(M):=[S,g,\bar{g}]\in \Omega_{r-1}(\mathbb RP^{r-1}\times M,\phi_M)$ in the 
normal bordism ring  $\Omega_*(\mathbb RP^{r-1}\times M,\phi_M).$   (If $S$ is empty, it is understood that $[S,g,\bar g]=0$.)
The element $\omega_{r}(M)$ is independent 
of the choice of $\underline{X}$.  Koschorke \cite[Theorem 13.3]{koschorke} showed that 
$\Span M\ge r$ if and only if $\omega_{r}(M)=0$.   We point out some important applications 
to span and stable span.  

The theorem below gives criterion for span to be at least $3$.  Koschorke considers 
{\it all} values of $\dim M\ge 7$, but we confine ourselves to the case when $\dim M\equiv 2\mod 4$.

\begin{theorem}\label{span3} {\em (U. Koschorke \cite[\S 14]{koschorke})}
Let $M$ be a $d$-dimensional manifold where $d\ge 10$.  Suppose that $\chi(M)=0, w_{d-2}(M)=0$ 
and that $d\equiv 2\mod 4$.  Then $\Span(M)\ge 3$.
\end{theorem}

\begin{theorem} \label{stabspan}{\em (U. Koschorke \cite[\S 20]{koschorke}, V. Eagle \cite{eagle}.)}
Let $M$ be a smooth compact connected manifold of dimension $d$. \\(a)  If $d\equiv 0\mod 2, $ and $\chi(M)=0$, 
then $\Span^0(M)=\Span(M)$. \\ 
(b) If $d\equiv 1\mod 4$ and if $w_1(M)^2=0$, then $\Span^0(M)=\Span(M)$ if 
the twisted Kervaire semi-characteristic $R_L(M)$ vanishes; if $R_L\ne 0$, then $\Span(M)=1$.
\\ (c) If $d\equiv 3\mod 8$ and $w_1(M)=w_2(M)=0$, then 
$\Span^0(M)=\Span(M)$ if $\hat{\chi}_2(M)=0$; if $\hat{\chi}_2(M)\ne 0$, then $\Span(M)=3$.
\end{theorem}

Koschorke had noted that the above results were obtained by Eagle in his PhD thesis using entirely different 
methods. 

We now state without proof the following theorem which determines the span of a 
stably parallelizable but non-parallelizable manifold.  

\begin{theorem}{\em (G. Bredon and A. Kosi\'nski \cite{bk}, E. Thomas \cite{thomas})} \label{b-k}
Let $M$ be a compact connected manifold of dimension $d$.  Suppose that 
$M$ is stably parallelizable. Then either $M$ is parallelizable or $\Span(M)=\Span(\bs^d)=\rho(d+1)-1$. 
If $d$ is odd and $d\notin\{ 1,3,7\}$, then $M$ is parallelizable if and only if the Kervaire semi-characteristic 
$\hat{\chi}_2(M)=0$. If $d$ is even, $M$ is parallelizable if and only if $\chi(M)=0$.  \hfill $\Box$
\end{theorem}

Note that $\mathbb{S}^d$ is parallelizable when $d=1,3,7$ although $\hat{\chi}_2(\mathbb S^d)\ne 0$.


\begin{remark} \label{lmp}{\em 
In view of Theorems \ref{b-k} and \ref{stabspan}, it is important to have criteria for the vanishing of the Kervaire    
semicharacteristics $\hat \chi_2(M)$ and $\kappa(M)$ of a compact connected orientable 
smooth manifold $M$ of dimension $d=2m+1$.   Note that the orientability assumption implies that the twisted 
semicharacteristic $R_L(M)$ equals $\kappa(M)$. 
Lusztig, Milnor and Peterson \cite{lmp} showed that $\hat \chi_2(M)-\kappa(M)$ equals the Stiefel-Whitney number 
$w_2w_{d-2}[M]\in \mathbb{Z}_2$.  In particular $\kappa(M)=\hat \chi_2(M)$ if $M$ is a spin manifold or if $M$ is 
null-cobordant.    Stong \cite{stong-cm} proved that if  $M$ admits a free smooth 
$\mathbb Z_2\times \mathbb Z_2$-action on $M$, then $\hat \chi_2(M)=0$.  
}
\end{remark}

We shall now give several examples, starting from elementary ones.


\begin{example}{\em 
(i) Let  $S$ be a compact orientable connected surface
of genus $g$. Its Euler-Poincar\'e characteristic 
is $\chi(S)=2-2g$.  Thus when $g\neq 1$, span of $S$ 
equals zero.  When $g=1$, $S$ equals the torus $\bs^1\times \bs^1$ 
which is parallelizable.  When $g=0,$ the surface $S=\bs^2$. 
Fix an imbedding $j:S\to \mathbb{R}^3$ and 
denote by $\nu$ the normal bundle (over $S$) with respect to $j$. Then
we obtain that $3\epsilon= j^*(\tau \mathbb{R}^3)=\tau S\oplus \nu$.  Since $S$ 
is orientable, the normal bundle $\nu$ is trivial and we conclude that 
$S$ is stably parallelizable.

(ii) Suppose that $M$ is a non-orientable surface.  Then it 
has an orientable double covering $p:S\to M$.  One has 
$\chi( M)=(1/2)\chi(S)$.  It follows that 
$\Span(M)=0$ except when $S$ is a torus $\bs^1\times \bs^1$.  
When $S$ is a torus,  $M$ is the Klein bottle and we have $\chi(M)=0$.  
By Hopf's theorem \ref{hopf}, we have $\Span(M)\geq 1$. Since $M$ is {\it not} orientable, $\Span(M)<\dim (M)=2$ and 
hence $\Span(M)=1$.  Since $M$ is non-orientable, it is not stably parallelizable. 

(iii) Any {\it orientable} compact connected manifold 
$M$ of dimension $3$ is parallelizable.   This was first observed by 
Stiefel. The proof involves obstruction theory and uses the fact that $\pi_2(SO(3))=\pi_2(\br P^3)=0$. See \cite[Problem 12-B]{ms}.

(iv) Let $M=S\times \mathbb{S}^1$, where $S$ is a non-orientable surface.  Then $M$ is non-orientable 
and hence not parallelizable.  So $1\le \Span(M)\le 2$.  If $S=K,$ the Klein bottle, we have $\Span(M)=2$.  
This follows from the isomorphism of vector bundles $\tau (M_1\times M_2)\cong \tau M_1\times \tau M_2
\cong pr_1^*(\tau M_1)\oplus pr_2^*(\tau M_2)$, where $pr_j:M\to M_j$ is the $j$-th projection.  
If $S=\mathbb{R}P^2$, or more generally if $S$ has odd Euler-Poincar\'e 
characteristic, then the Stiefel-Whitney class $w_2(S)\in H^2(S;\mathbb{Z}_2)\cong \mathbb{Z}_2$ is non-zero.  
This implies that $w_2(M)\ne 0$.  It follows that $\Span(M)=1$.

(v)  Becker \cite{becker} has determined the 
span of quotient spaces $M:=\Sigma^{n-1}/G$ when $n\ne 8,16$, where $G$ is a finite group that acts freely and smoothly 
on a homotopy sphere $\Sigma^{n-1}.$ 
In the special case when $\Sigma^{n-1}$ is the standard sphere and $G$ is a group of odd order that 
acts orthogonally, it was shown by Yoshida \cite{yoshida} that $\Span(M)=\Span(\mathbb{S}^{n-1})=\rho(n)-1$ for 
$n\ne 8$, settling a conjecture of Sjerve \cite{sjerve}.   The following general result which is of independent interest 
and also due to Becker, is a crucial step in the determination of span of $M$.
}  
\end{example}

\begin{theorem} {\em (Becker)}
Suppose that $N$ is a compact connected orientable smooth $d$-dimensional manifold and $\wt{N}\to N$ is covering projection 
of {\em odd} degree.  Let $k\le (d-1)/2$ is a positive integer.  Then $\Span(N)\ge k$ if and only if $\Span(\wt{N})\ge k$.  
\end{theorem}

\subsection{Span of products of manifolds}
If $M,N$ are compact connected smooth manifolds, then it is clear that $\Span(M)+\Span(N) \le 
\Span(M\times N)$ and that equality holds when $\Span(M)=0=\Span (N)$ by Hopf's Theorem \ref{hopf}.
If $M,N$ are stably parallelizable, then so is $M\times N$.   The converse is also valid;
this is because 
the normal bundle to the inclusion of each factor into the product $M\times N$ is trivial.  
On the other hand, if $M\times N$ is parallelizable, we cannot conclude that $M$ and $N$ are 
parallelizable.   See Theorem \ref{product-sphere} below.
There is no general `formula' that expresses 
$\Span(M\times N)$ in terms of the span of $M,N$.  In this section we shall obtain some 
bounds for the span of $M\times N$, when one of the factors is stably parallelizable.  

We begin with the following result whose proof, due to E. B. Staples \cite{staples}, is surprisingly simple, considering that the solution to the vector field 
problem for spheres is highly non-trivial.  

\begin{theorem} \label{product-sphere} {\em (Staples \cite{staples}.)} 
The manifold $\bs^m\times \bs^n$ is parallelizable 
if at least one of the numbers $m,n\geq 1$ is odd.  
If both $m,n$ are even, then $\Span(\bs^m\times \bs^n)=0$.
\end{theorem}
\begin{proof}
Assume that $m$ is odd. Then $\tau(\bs^m)=\epsilon\oplus \eta$ for some subbundle $\eta\subset \tau(\bs^m)$.  Let $p_i$ denote the projection to the $i$-th factor 
of $\bs^m\times \bs^n$. Using $\epsilon\oplus \tau(\bs^n)=(n+1)\epsilon$, we have 
$\tau(\bs^m\times \bs^n)= p_1^*(\tau(\bs^m))\oplus 
p_2^*(\tau(\bs^n))=p_1^*(\eta)\oplus \epsilon\oplus \tau(\bs^n)=p_1^*(\eta)\oplus (n+1)\epsilon=\\p_1^*(\eta\oplus 2\epsilon)\oplus (n-1)\epsilon=(m+1)\epsilon\oplus (n-1)\epsilon=(m+n)\epsilon$.

If both $m,n$ are even, then $\chi(\bs^m\times \bs^n)=\chi(\bs^m)\times \chi(\bs^n)=4$ and so by Hopf's Theorem \ref{hopf} $\Span(\bs^m\times \bs^n)=0$.
\end{proof}

Note that in the above proof we exchanged, repeatedly, $\tau \mathbb{S}^n\oplus \epsilon$ for $(n+1)\epsilon$.  This 
is often referred to as {\it boot-strapping}.  We will have several occasions in the sequel to 
use it.  The next theorem, essentially due to Staples, is a generalization of the above.    

\begin{theorem} \label{product-span} {\em (Staples \cite{staples}.)}
Suppose that $M,N$ are smooth compact connected positive dimensional manifolds.  
Assume that $\chi(N)=0$ and that $\Span^0(M)\ge 1$.  Then 
$\Span^0(M)+\Span^0(N)\le \Span(M\times N)\le \Span^0(M\times N)
\le \min\{ \Span^0(M)+\dim N, \dim M+\Span^0(N)\}.$
Moreover, if $M, N$ are stably parallelizable, then $M\times N$ is parallelizable.  
\end{theorem}
\begin{proof}
By Hopf's theorem \ref{hopf}, $r:=\Span(N)\ge 1$.  Write $\tau N=\theta\oplus r\epsilon$ and $\tau(N)\oplus \epsilon=\eta\oplus 
(s+1)\epsilon$ where $s:=\Span^0(N)$.  Similarly write $\tau M\oplus \epsilon=\xi \oplus (p+1)\epsilon, $
where $p=\Span^0(M)\ge 1$.  
To simplify notations, we will denote the pull-back bundle $pr_1^*(\tau M)$ also by the same symbol $\tau M $ 
where $pr_1:M\times N\to M$ is the first projection.  Similar notational conventions will be followed for 
$pr_2^*(\theta), pr_2^*(\tau N)$, etc. 

We have the chain of bundle isomorphisms by boot-strapping:
\[\begin{array}{rcl}
\tau (M\times N) &= &\tau M \oplus \tau N \\
&=& \tau M\oplus \theta\oplus r\epsilon  \\
&=& (\tau M\oplus \epsilon)\oplus \theta \oplus (r-1) \epsilon \\
& =& \xi\oplus (p+1)\epsilon\oplus \theta\oplus (r-1)\epsilon  \\
&= & \xi \oplus (p-1)\epsilon \oplus \theta \oplus (r+1)\epsilon  \\
&= & \xi \oplus (p-1)\epsilon\oplus \tau N\oplus \epsilon  \\ 
& = &\xi\oplus (p-1)\epsilon \oplus \eta\oplus (s+1) \epsilon\\
&=&\xi\oplus \eta\oplus (p+s)\epsilon.\\
\end{array} \]

Hence $\Span(M\times N)\ge p+s=\Span^0(M)+\Span^0(N)$. 

Fix a point $q\in N$. The natural inclusion $M\subset M\times N$ pulls back the tangent bundle of $M\times N$ to 
$\tau M\oplus (\dim N)\epsilon$.  This implies that $\Span^0(M\times N)$ cannot exceed $\Span^0(M)+\dim N$.   
Similarly $\Span^0(M\times N)\le \Span^0(N)+\dim M$.
\end{proof}

\begin{example}{\em 
(i)  Let $M=\mathbb RP^m\times \mathbb S^n$.   
 If $m$ is odd, $\Span(M)=\Span^0(\mathbb RP^m)+n=\rho(m+1)+n-1.$ For the last equality, see Theorem \ref{proj-stablespan}. If $m,n$ are both 
even, $\Span(M)=0, \Span^0(M)=n$, since $w_m(M)\ne 0$.  If $m$ is even and $n$ odd, then $\Span (\mathbb S^n)\le \Span(M)\le \Span^0(M)= n$. But the exact value of span seems to be unknown in general.   
When $m=2$ and $n\equiv 1 \mod 8$, $n\ge 9$, it turns out that $\Span (M)=3$ whereas $\Span^0(M)=n$; see   
\cite[Exercise 20.18]{koschorke}. 

(ii) Suppose that $M$ is the boundary of a parallelizable manifold-with-boundary $W$.  Then $M$ is stably parallelizable. 
This is because $W$ is necessarily orientable and the normal bundle $\nu$ to the inclusion $M\hookrightarrow W$ 
is a trivial line bundle. (One may take `outward pointing'  unit normal at each point of $M$ with respect to a Riemannian 
metric on $W$.)   Hence $\tau W|_M\cong \tau M\oplus \epsilon$.  Since $\tau W$ is trivial, so is 
 $\tau M\oplus \epsilon$, that is, $M$ is stably parallelizable.  
 
(iii) A well-known result of Kervaire and Milnor \cite[Theorem 3.1]{km} says that any smooth homotopy sphere is stably 
parallelizable. 
An immediate corollary is that a product of two or more smooth homotopy spheres is parallelizable 
if and only if at least one of them is odd-dimensional.  If all the homotopy spheres are of even dimension, then 
their span is zero (in view of Hopf's theorem \ref{hopf}). 

(iv) J. Roitberg \cite{roitberg} has constructed smooth $(4k-2)$-connected 
manifolds $M_1, M_2$ of dimension $d=8k+1$ for each  
$k\ge 2,$ having the following properties: (a)  $M_1$ and $M_2$ are homeomorphic (in fact $M_1,M_2$ admit 
PL-structures and are PL-homeomorphic) with $\Span(M_1)=1=\Span(M_2)$, 
(b) $M_2$ is stably parallelizable, but $M_1$ is not, (c)   $\Span(M_2\times N)> \Span(M_1\times N)$ for any 
stably parallelizable manifold $N$ of dimension $n\ge 1$.   The construction of the manifolds $M_1,M_2$ 
involves deep machinery which goes far beyond the scope of these notes.   We shall be content with some 
remarks.  It turns out that $M_2$ has the same $\mathbb{Z}_2$-homology groups as the sphere $\mathbb{S}^m$.
It follows that the Betti numbers $b_j(M_2)$ of $M_2$ vanish for $1\le j\le d-1=8k$.  Hence the real Kervaire 
semi-characteristic $\kappa(M_2)=\sum_{0\le j\le 4k}b_j(M_2)=1$. (See Remark \ref{lmp}.)  By Theorem \ref{span2}(a), $\Span(M_2)\le 1$. Since 
$M_2$ is odd-dimensional, equality must hold.  Since $M_1$ is homeomorphic to $M_2$, the same argument 
applies to $M_1$ as well and so $\Span(M_1)=1$.  
For the assertion (c), note that $M_2\times N$ is parallelizable 
(by Theorem \ref{product-span}) but $M_1\times N$ is not stably parallelizable since $M_1$ is not.

(v) Crowley and Zvengrowski \cite{cz} have extended the results of Roitberg to dimensions $\ge 9$.  More precisely, 
for each $d\ge 9$, they have shown the existence of manifolds $M_1, M_2$ which are PL-homeomorphic but 
$\Span(M_1)\ne \Span(M_2)$.  They also showed 
that there can be no such examples in dimensions up to $8$.
}
\end{example}

In contrast to the case of spheres, the 
span of the product $M=\br P^{m-1}\times \br P^{n-1}$ of real projective spaces for general $m,n$ is unknown.  
Of course, the span is zero when both $m,n$ are odd since in that case $\chi(M)=1$.  Using the formula for 
Stiefel-Whitney classes of projective spaces, one obtains 
that $w_{m+n-k-l}(M)=w_{m-k}(\mathbb{R}P^{m-1})\times w_{n-l}(\mathbb{R}P^{n-1})\ne 0$ where $k=2^r,l= 2^s$ are
highest powers of $2$ which divide $m, n$ respectively.   
 It follows that $\Span(M)\le k+l-2$.  This is uninteresting when $m,n$ are both powers of $2$ and is 
 strong when $k, l\le 2$. 
 On the other hand one has the lower bound $\Span(M)\ge \Span^0(\mathbb RP^{m-1})+\Span^0(\mathbb RP^{n-1})=
 \rho(m)+\rho(n)-2$.  (See Theorem \ref{proj-stablespan}.)

The following result is due to Davis.
\begin{theorem} {\em (Davis \cite{davis}) }  Suppose that $16\not\!|n$ and $16\not \!|m$, or, $m=2,4,8.$  Then
$\Span (\mathbb{R}P^{m-1}\times \mathbb{R}P^{n-1})=\rho(m)+\rho(n)-2$.  
\end{theorem}
\begin{proof}
Let $m=2,4,$ or $8$.   Then $\rho(m)=m$ and for {\it any} manifold $N$ we have $\Span(\mathbb{R}P^{m-1}\times N)=m-1+
 \Span^0(N)$. Taking $N=\mathbb{R}P^{n-1}$, we get $\Span^0(N)=\Span(N)=\rho(n)-1$ which proves the 
 assertion in this case.
 
 Suppose that $16$ divides neither $m$ nor $n$.  In this case, $\rho(m)=2^r, \rho(n)=2^s$ where $m=2^r\cdot m', n=2^s\cdot n'$ 
 with $m',n'$ being odd.  Using the formula $w_j(\mathbb{R}P^{m-1})={m\choose j}a^j\in H^j(\mathbb{R}P^{m-1};\mathbb{Z}_2)=\mathbb{Z}_2a^j$, we obtain that $w_{ m-2^r}(M)\ne 0$, and $w_j(M)=0$ for $j>m-2^r$.  Similar statement 
 holds for $N$ and so we obtain that $w_{m+n-2^r-2^s}(M\times N)=w_{m-2^r}(M)\times w_{n-2^s}(N)\ne 0$. 
 Hence $\Span(M\times N)\le 2^r+2^s-2$. Since $\Span(M\times N)\ge \Span(M)+\Span(N)=\rho(m)-1+\rho(n)-1=
 2^r+2^s-2$, the assertion follows.
 \end{proof}
 
When $16|m$, the Stiefel-Whitney upper bound is rather weak.  Using BP-cohomology, Davis \cite{davis} has obtained an upper bound for the span of $\mathbb{R}P^{m-1}\times \mathbb RP^{n-1}$  
which is sharper than the previously known ones.     
No example of a pair of numbers $(m,n)$ seems to be known where the span of $\mathbb{R}P^{m-1}\times \mathbb{R}P^{n-1}$ is {\it strictly} bigger 
than the Radon-Hurwitz lower bound $\rho(m)+\rho(n)-2$.    

\section{Vector fields on homogeneous spaces}
 In this section 
we shall consider the vector field problem for 
homogeneous spaces, mostly focusing on Stiefel manifolds, Grassmann manifolds, and related 
spaces.  We will assume familiarity with Lie groups and representation theory 
of compact Lie groups. As we proceed further, acquaintance with (topological) K-theory will also 
be assumed.  

Let $G$ be any Lie group and let $H$ be a closed subgroup.  We consider the natural differentiable structure 
on the homogeneous space $M=G/H$. Thus the quotient map $G\to G/H$ is smooth.   

We begin with the following well-known result.  (Compare Example \ref{basicexamples}(v).)

\begin{theorem} \label{bh}  {\em (Borel-Hirzebruch \cite{bh}.)}
Let $\Gamma\subset G$ be a discrete subgroup of a Lie group $G$.  Then $G/\Gamma$ is parallelizable.    
\end{theorem}
\begin{proof}  We shall work with the space of right 
cosets $\Gamma\backslash G$ instead of $G/\Gamma$.
The tangent bundle $\tau (\Gamma\backslash G)$ has the following 
description: $T(\Gamma\backslash G)=G\times \frak{g}/\sim$ where 
$(x,v)\sim(hx, d\lambda_h(v)), x\in G, v\in \frak{g}, h\in \Gamma$.  Let $v_1,\ldots, v_n$ be everywhere linearly independent $G$-invariant vector fields on $G$ where $n=\dim G$. 
Since $d\lambda_h(v_j(x))=v_j(hx)$ for all $h\in H$, we see 
that $\bar{v}_j(\Gamma x)=[x,v_j(x)]\in T_{Hx}(\Gamma\backslash G)  $ is a well-defined 
(smooth) vector field on $\Gamma\backslash G$ for $1\leq j\leq n$.   See Example \ref{basicexamples}(v).
\end{proof}

\subsection{Homogeneous spaces of compact Lie group.}
In this section we consider 
homogeneous spaces $G/H$ where $G$ is a compact 
connected Lie group and $H$ a closed subgroup.

First suppose that $T$ is a maximal torus of a compact connected Lie group $G$.   
That is, $T\subset G$ is isomorphic to $(\bs^1)^r$ with $r$ largest. The number $r$ is called the rank of $G$. 
It is well-known that $G$ is a union of its maximal tori and 
 that any two maximal tori in $G$ are conjugates in $G$.
 Let $N_G(T)$ denote the normalizer of $H$ in $G$.  Then $W(G,T):=N_G(T)/T$ is a finite group known as 
the Weyl group of $G$ with respect to $T$.  It is known 
that the Euler-Poincar\'e characteristic of $G/T$ equals $|W(G,T)|$, the cardinality of the Weyl group.   
To see this, first note that an element $gT$ is a $T$-fixed point for the action of $T$ on 
$G/T$ if and only if $g\in N(T)$.  Since $g_0T=g_1T$ if and only if $g_0^{-1}g_1\in T$, we have a 
bijection between $T$-fixed points of $G/T$ and $W(G,T)$.  
Applying \cite[Theorem 10.9]{bredon} we see that 
$\chi(G/T)=|W(G,T)|$.  If $H$ is a closed connected subgroup of $G$ such that $T\subset H\subset G$, then 
$W(H,T)$ is a subgroup of $W(G,T)$ and the coset space $W(G,T)/W(H,T)$ will be denoted $W(G,H)$. 

Let $H\subset G$ be any connected subgroup having the {\it same} 
rank as $G$.  If $T\subset H$ is a maximal torus of $H$, then 
\[\chi(G/H)=|W(G,H)|. \eqno(2)\]  
To see this, observe that one has a fibre bundle with fibre space $H/T$ and projection $G/T\to G/H$.  The required result then 
follows from the multiplicative property of the Euler-Poincar\'e characteristic and the formula for $\chi(G/T)$. 

Suppose that $S\cong (\bs^1)^s$ is a toral subgroup of $G$ where $s<r=\rank(G)$.  Then $S$ is properly contained 
in a maximal torus $T$ of $G$. Considering the 
fibre bundle with fibre $T/S\cong (\bs^1)^{r-s}$ and projection $G/S\to G/T$, we see that $\chi(G/S)=\chi(G/T)\cdot \chi(T/S)=0$ since 
$\chi(T/S)=(\chi(\bs^1)^{r-s})=0$.  
It follows that if the rank of $H$ is less than the rank of $G$, 
then, taking $S$ to be a maximal torus of $H$ and using the $H/S$-bundle with projection $G/S\to G/H$,  
we have $\chi(G/H)=\chi(G/S)/\chi(H/S)=0.$   If $H\subset G$ is not connected, denoting the identity 
component $H_0$, the natural map $G/H_0\to G/H$ is a covering projection of degree $|H/H_0|$ 
and so $\chi(G/H)=\chi(G/H_0)/|H/H_0|$.  
The following result is an immediate consequence of Hopf's Theorem \ref{hopf}.

\begin{theorem} \label{spanvanishing}
Let $G$ be a compact connected Lie group and let $H$ be a closed subgroup of $G$.  Let $H_0$ denote the 
identity component of $H$.  
If $\rank(H_0)=\rank(G)$, then $\Span(G/H)=0$.  If $\rank(H_0)<\rank(G),$ 
then $\Span(G/H)>0.$ \hfill $\Box$
\end{theorem}

Next we describe the tangent bundle of $G/H$ in terms 
of the adjoint representation.  
We do not assume that $G$ is compact.  Also $H$ is not assumed to be connected.

The conjugation $g\mapsto \iota_g$ defined as $\iota_g(x)=gxg^{-1}, x\in G,$  defines an action of $G$ on itself. 
Clearly the identity element $e\in G$ is fixed under this action.  Hence, we obtain a representation $G\to 
GL(\frak{g})$ defined as $g\mapsto d\iota_g|_e$.  This 
is referred to as the adjoint representation, denoted $Ad_G.$  
By restricting the action to the subgroup $H\subset G$ we 
obtain a representation of $Ad_G|_H$. Note that since $H\subset G$, the adjoint representation of $H$ on $T_eH=\frak{h}$ is a subrepresentation of $Ad_G|_H$ and moreover we obtain a 
representation of $H$ on $\frak{g}/\frak{h}$.  Further, the tangent space to $G/H$ at the identity 
coset $H$ may be identified with $\frak{g}/\frak{h}$, as 
can be seen by considering the differential $d\pi|_e:\frak{g}=T_eG\to T_{\pi(e)}(G/H)$ of the projection of the $H$-bundle 
$\pi:G\to G/H$, whose kernel is $\frak{h}=T_eH$.  It turns out that the tangent bundle $\tau(G/H)$ has the following description: 
\[T(G/H)=G\times_H \frak{g}/\frak{h} \eqno(3)\]
 where the right hand side denotes the 
 quotient of $G\times\frak{g}/\frak{h}$ by the relation 
 $(g,v+\frak{h})\sim(gh^{-1}, Ad(h)(v)+\frak{h})$.  The projection $G\times _H\frak{g}/\frak{h}\to G/H$ defined 
 as $[g, v+\frak{h}]\mapsto gH\in G/H$ is the projection 
 of a vector bundle with fibre $\frak{g}/\frak{h}$ which is isomorphic to the tangent bundle of $G/H$.

 The exact sequence of $H$-representations 
\[0\to \frak{h}\to \frak{g}\to\frak{g}/\frak{h}\to 0\]
induces an exact sequence of vector bundles over $G/H$: 
\[0\to \nu\to \mathcal{E}\to \tau_{G/H}\to 0\eqno(4)\]
where $\mathcal{E}=G\times_H\frak{g}$.  The bundle $\mathcal{E}$ is isomorphic 
to the trivial bundle $d\epsilon$ of rank $d:=\dim G$ since the action of $H$ on $\frak{g}$ 
extends to an action of $G$ on $\frak{g}$ (namely 
the adjoint action).

 \begin{proposition} \label{torus}
 Let $H\subset G$ be a toral subgroup of a compact connected 
 Lie group $G$.  Then $G/H$ is stably parallelizable; it is parallelizable if and only if $rank(G)>\dim H$. 
 \end{proposition}
\begin{proof}
The bundle $\nu$ with total space 
$G\times _H\frak{h}$ is trivial since the adjoint representation of $H$ is trivial.  (This is because 
the group $H\cong (\bs^1)^s$ is abelian.) 
The above sequence of vector bundles splits (after choosing a Euclidean metric on $\mathcal{E}$), and so we have
$d\epsilon\cong\tau(G/H)\oplus \nu=\tau(G/H)\oplus s\epsilon$ where $s=\dim H$.  This proves our first assertion.   
To prove the last assertion, note that if $\dim H=\rank(G)$, then 
$\chi(G/H)=|W(G,H)|\neq 0$ and so $\Span(G/H)=0$. 
If $s=\dim H<\rank (G)=:r$, choose a maximal torus $T\supset H$.  Consider the $T/H$-bundle with projection $\pi:G/H\to G/T$.  
This is a {\it principal} bundle with fibre and structure group $T/H$.  Hence the vertical bundle is trivial (see Example \ref{basicexamples}(iv)) and we have 
 $\pi^*(\tau(G/T))\oplus (r-s)\epsilon\cong \tau(G/H)$.
 Since $\tau(G/T)$ is stably trivial and $r>s$, it follows that $\tau G/H$ is trivial, i.e., $G/H$ 
 is parallelizable.
\end{proof}

Let $\psi:H\to GL(V)$ be a representation of a Lie group $H$ on a real vector space $V$. Suppose that $H$ is a closed 
subgroup of a Lie group $G$.  Denote by $\alpha(\psi)$ the associated vector bundle $G\times_HV\to G/H$.  (Here $G\times_HV=G\times V/\!\sim$ where $(gh, v)=(g, \psi(h)(v)), ~\forall (g,v)\in G\times V, h\in H$.)  
Bundles over $G/H$ associated to representations of $H$ are referred to as {\it homogeneous vector bundles}.  
For example, as we have seen already, the tangent bundle 
$\tau (G/H)$ is a homogeneous vector bundle associated to the representation on $\mathfrak{g/h}$ induced by $Ad_G|_H$, 
the adjoint representation of $G$ 
(on $\mathfrak{g}$) restricted to $H$ and the adjoint representation $Ad_H$ (on $\mathfrak{h}\subset \mathfrak{g}$). 
This so-called $\alpha$-construction defines a ring homomorphism $\alpha:RO(H)\to KO(G/H)$ from the real representation ring of 
$G$ to the $KO$-theory of $G/H$.  Analogously, one has the $\alpha$-construction on complex 
representations leading to $\alpha_\mathbb{C}:R(H)\to K(G/H).$   The kernel of $\alpha$ (resp. 
$\alpha_\mathbb{C}$) contains the ideal of $RO(H)$ (resp. $R(H)$) generated by the elements of the form $[E]-\dim [E]$ 
where $E$ is the restriction to $H$ of a real (resp. complex) representation of $G$. 

Denoting the complexification 
homomorphisms $RO(H)\to R(H)$ and  $KO(G/H)\to K(G/H)$ by the same symbol $c$, one has $c\circ \alpha=\alpha_\mathbb{C}\circ c$.  Similarly we have the `realification' homomorphisms $r:R(H) \to RO(H)$ and $r:K(G/H)\to KO(G/H)$ which forgets the complex structure.  Note that $c$ is a ring homomorphism whereas $r$ is only a homomorphism 
of abelian groups.  One has $r\circ c=2$ and $c\circ r=1+\bar{},$ where the notation $~\bar~$ stands for the complex conjugation.  These relations hold on the real and complex representation rings and also on the real and complex K-theoretic rings.  
We refer the reader to \cite{ah} for detailed discussion and further results on the relation between representation 
rings $G, H$ and the K-theory $G/H$.

Singhof and Wemmer \cite{sw} established Theorem \ref{adjoint} given below.  
The sufficiency part is immediate from the exact sequence (4) 
of vector bundles and has been noted earlier (see \cite[p. 103]{singhof}.)   The proof of the 
necessity part involves verification using the classification of compact simple Lie groups.   Recall that a connected Lie 
group $G$ is said to be {\it simple} if $G$ is not abelian and has no proper connected normal subgroups.  For example, 
$SU(n)$ is simple, although its centre is a cyclic group of order $n$.  
One says that $G$ is {\it semisimple} if its universal cover is a product of simple Lie groups.   
A compact connected Lie group is semisimple if and only if its centre is finite. 

One has also the Grothendieck group $RSp(G)$ of (virtual) $G$-representations of left $\mathbb{H}$-vector spaces.  The restriction 
homomorphisms $RO(G)\to RO(H), R(G)\to R(H), RSp(G)\to RSp(H)$ will all be denoted by the same symbol $\rho$. 
Note that $\rho$ is a ring homomorphism in the case of real and complex representation rings.  
Although $RSp(G)$ 
is only an abelian group, one can form the tensor product of a right and a left 
$\mathbb H$-representation to obtain a real representation.   If $W$ (resp. $U$) is a left (resp. right) $\mathbb H$-vector space, then 
$U\otimes_{\mathbb H} W$ has only the structure of a real vector space of dimension $ 4\dim_{\mathbb H}U\dim_{\mathbb H}W$. 
If $H$ acts on $U,W$ ~ $\mathbb H$-linearly, 
then $U \otimes_{\mathbb H} W$ is naturally a real representation of $H$.  
Its isomorphism class determines an element, denoted  
$[U\otimes_{\mathbb H} W]$, in $RO(H)$. 
If  $V$ is a left $\mathbb H$-vector space, denote by $V^*$ the right $\mathbb H$-vector 
space where $v\cdot q=\bar{q}v, v\in V, q\in\mathbb H$.   
We have a $\mathbb Z$-bilinear map $\beta: RSp(H)\times RSp(H) \to RO(H)$ 
defined as $([V],[W])\mapsto [V^*\otimes_{\mathbb H} W] $.  If $V=\mathbb H$ is a trivial $H$-representation, 
then $\beta([V],[W])= [W_\mathbb{R}]$ where $W_\mathbb{R}$ stands for the same $H$-representation 
$W$ with scalar multiplication restricted to $\mathbb R\subset \mathbb H$. 

We denote by $J=J(G,H)$ the ideal of $RO(H)$ 
generated by elements of the form (i) $\rho(x)-\dim x, x\in RO(G)$, (ii) $\beta(\rho(x-\dim_{\mathbb H}(x)[\mathbb H]), y), x\in RSp(G), 
y\in RSp(H)$.  It is easy to see that, if $x\in RO(G)$, then $\rho(x)-\dim x$ is contained in the kernel of 
$\alpha:RO(H)\to KO(G/H)$.  In fact we have $J(G,H)\subset \ker(\alpha)$.

\begin{theorem} {\em (See \cite{sw}, \cite{sw-err}) } \label{adjoint}
(i) Let $G$ be a simply connected compact connected Lie group and $H$ a closed connected subgroup.  
Then $G/H$ is stably parallelizable if $[Ad_H]$ is in the image of the restriction homomorphism 
$\rho:RO(G) \to RO(H)$.  \\
(ii) Conversely, suppose that $G/H$ is stably parallelizable and that $G$ is simple. (a) If $G\ne Sp(n)$, then 
$[Ad_H]$ is in the image of $\rho$.  
(b) If $G=Sp(n)$, then $[Ad_H]-\dim H$ is in the ideal $J(Sp(n), H)$ of $RO(H)$. \hfill $\Box$ 
\end{theorem}

The first part of the above theorem holds for any connected Lie group.  If $G$ acts linearly on a real 
vector space $W$, then the associated vector bundle $\alpha(W)$ on $G/H$ with projection $G\times_H W\to G/H$ is 
trivial, without any condition on $G$.  If the $H$ action on $\mathfrak h\oplus \mathbb R^k=W$ (where the adjoint action on 
the first summand and the trivial action on $\mathbb R^k$ is understood) extends to a linear action of $G$, then 
$\nu\oplus k\epsilon\cong \alpha(W)$ is a trivial vector bundle on $G/H$ where $\nu$ is as in the exact sequence (4).  
It follows that $\tau G/H$ is trivial.

We state, without proofs, the following results of Singhof \cite{singhof}.

\begin{theorem} {\em (Singhof \cite{singhof}.)}\label{torussimple}
Let $G$ be a connected compact simple Lie group and let $H$ be a closed connected subgroup of $G$ such that 
$H$ is neither a torus nor semisimple.  Then the first Pontrjagin class $p_1(G/H)$ is non-zero.  
In particular, $G/H$ is not stably 
parallelizable.  \hfill $\Box$
\end{theorem}

\begin{theorem}{\em (Singhof \cite{singhof})} \label{unitarygroupquotients}
Let $H\cong SU(k_1)\times \cdots\times SU(k_r)$ be a closed subgroup of $G=SU(n)$.  Then the following are equivalent:\\
\indent (i) $G/H$ is stably parallelizable.\\
\indent (ii) $H$ equals one of the following subgroups: 
(a) $k_j=2, 1\le j\le r\le n/2$ and $H$ is embedded block diagonally, 
(b) $n=4, H=SU(2)$ is the diagonal copy of $SU(2)\times SU(2)\subset SU(4)$, 
(c) $H=SU(k)$, with standard embedding.  
Moreover, if $SU(n)/H$ is stably parallelizable and is not a sphere, then it is parallelizable. \hfill $\Box$ 
\end{theorem}

Singhof and Wemmer \cite{sw} completely determined all pairs $(G,H)$ where $H$ is a closed connected subgroup of a 
compact simply connected simple Lie group $G$ such that $G/H$ is (stably) parallelizable.

Let $H=SU(k_1)\times \cdots\times SU(k_r)\subset SU(n), r\ge 2,$ with $k:=\sum k_j< n$.  Set $M:=SU(n)/H, N:=SU(k)/H$ and 
$B:=SU(n)/SU(k)$.  
One has a fibre bundle with fibre space $N$ and projection  $p:M\to B$ since $H\subset SU(k)$.  
Since $k<n$, the base 
space $B$ is the complex Stiefel manifold if $k\le n-2$ and is the sphere $\mathbb S^{2n-1}$ if $k=n-1$.   
In any case $B$ is stably parallelizable.  (The parallelizability results for Stiefel manifolds will be discussed 
in detail in \S \ref{stiefelmanifolds}.)
 Denote by $F$ the complex flag manifold $SU(k)/K$ where $K:=S(U(k_1)\times \cdots\times U(k_r))$.  
We have the following result due to Sankaran (unpublished).

\begin{theorem}  \label{quotsofsun}
With notations as above, let $r\ge 2$ and let $k=\sum_{1\le j\le r}k_r$.  Then: \\
(i) $\Span(N)\ge r-1$.\\
(ii) If $k<n$, then 
$r-1+n^2-k^2\le \Span(M)\le r-1+n^2-k^2+\Span^0(F)$;\\
 in particular, if $\chi(F)=k!/(k_1!\cdots k_r!) $ is odd, then $\Span(M)=r-1+n^2-k^2$. 
\end{theorem}

\begin{proof}  We shall only obtain the bounds for span of $M=SU(n)/H$.   
Let $V:=SU(n)/K$ where  $K=S(U(k_1)\times 
\cdots\times U(k_r))$.
One has a principal fibre bundle $\pi: M \to V$ with fibre and structure group the torus 
$K/H\cong (\mathbb S^1)^{r-1}$.   Hence we see that $\tau M=(r-1)\epsilon \oplus 
\pi^*(\tau SU(n)/K)$.   
Again $q:V\to B$ is a fibre bundle projection with fibre $F:=SU(k)/K$ and so, by Example \ref{basicexamples}(iv) we obtain 
a splitting $\tau(SU(n)/K)\cong q^*(\tau B)\oplus \eta$ where $\eta$ restricts to the tangent bundle of 
$F$ along any fibre of $q$. 
Hence 
\[\begin{array}{rcl}
\tau (M)&=&(r-1)\epsilon\oplus \pi^*(q^*(\tau B))\oplus \pi^*(\eta)\\
&\cong &(r-1+\dim B)\epsilon \oplus \pi^*(\eta)\\
\end{array}\]
 since $B$ is stably parallelizable and $r\ge 2$.   
Therefore $r-1+n^2-k^2\le  \Span(M)\le 
r-1+n^2-k^2+\Span^0(F)$ as $\dim B=n^2-k^2$.  Finally, if $\chi(F)$ is odd, then $\Span(F)=\Span^0(F)=0$ as the top Stiefel-Whitney class of $F$ is non-zero and so the last assertion follows.  
Note that the equality $\chi(F)=k!/(k_1!\cdots k_r!)$ follows from (2) and the fact that  the Weyl group of $SU(k)$ is the permutation 
group $S_k$.
\end{proof}

Next we shall discuss some important special cases of compact homogeneous spaces. 

\subsection{Stiefel manifolds} \label{stiefelmanifolds}
Let $1\leq k<n$.  Recall that the Stiefel manifold $V_{n,k}$ is the space of all ordered $k$-tuples $(v_1,\ldots,v_k)$ 
of unit vectors in $\mathbb{R}^n$ which are pairwise orthogonal (with respect to the standard inner product).  
When $k=1$, $V_{n,1}$ is the 
sphere $\bs^{n-1}$.  The group $SO(n)$ acts transitively 
on $V_{n,k}$ with isotropy at $(e_1,\ldots,e_k)$ being 
$I_k\times SO(n-k)=SO(n-k)$. Hence $V_{n,k}\cong SO(n)/SO(n-k)$.  
The complex and quaternionic Stiefel manifolds are defined analogously using the standard Hermitian product 
on $\mathbb{C}^n$ and the standard `quaternionic' product $\mathbb{H}^n$ defined as $q\cdot q'=\sum_{1\le r\le n} \bar{q}_rq'_r, q,q'\in \mathbb{H}^n$. 
We have the following description of $W_{n,k}, Z_{n,k}$ 
as coset spaces: $W_{n,k}\cong U(n)/U(n-k)=SU(n)/SU(n-k)$ and $Z_{n,k}\cong Sp(n)/Sp(n-k)$.  Note that $V_{n,1}=\mathbb{S}^{n-1}, W_{n,1}=\mathbb{S}^{2n-1}, Z_{n,1}=\mathbb{S}^{4n-1}$.  
We call an element of $V_{n,k}, W_{n,k}, Z_{n,k}$ an orthonormal, hermitian, quaternionic {\it $k$-frame} (or more briefly 
a $k$-frame) respectively.

Let $\beta_{n,k}$ (or more briefly $\beta$) denote the real vector bundle over $V_{n,k}$ whose fibre over any 
$k$-frame  $v=(v_1,\ldots, v_k)\in V_{n,k}$ is the real vector space $\{v_1,\ldots, v_k\}^\perp\subset \mathbb{R}^n$. 
The complex vector bundle of rank $n-k$ over $W_{n,k}$ and the quaternionic (left) vector bundle of rank $n-k$ 
over $Z_{n,k}$ are defined similarly.  One has the $\mathbb{F}$-vector bundle isomorphism 
\[k\epsilon_\mathbb{F}\oplus \beta_{n,k}\cong n\epsilon_\mathbb{F} \eqno(5) \] 
where $\mathbb{F}=\mathbb{R},\mathbb{C},\mathbb{H}$ according as the base space is 
$V_{n,k}, W_{n,k}, Z_{n,k}$; here $\epsilon_\mathbb{F}$ denotes the trivial $\mathbb{F}$-vector 
bundle.  As always, $\epsilon$ would denote the trivial {\it real} line bundle.

\begin{theorem} {\em (W. Sutherland \cite{ws}, K. Y. Lam \cite{lam}, D. Handel \cite{handel}.)}
The real, complex, and quaternionic Stiefel manifolds $V_{n,k}, W_{n,k}, Z_{n,k}$ are parallelizable when $k\ge 2$.  
\end{theorem}
\begin{proof}
We shall only consider the case of the real Stiefel manifolds;
The following description of the tangent bundle is due to Lam: 
\[\tau V_{n,k}\cong k\beta\oplus {k\choose 2}\epsilon, ~~\tau W_{n,k}\cong k^2\epsilon\oplus 2k\beta, ~~\tau Z_{n,k}\cong 
(2k^2+k)\epsilon\oplus 4k\beta, \eqno(6)\]
where the isomorphisms are, of course, of real vector bundles; by abuse of notation,  
$\beta$ stands for the underlying real vector bundle (in the complex and quaternionic cases).  
If $k\ge 3$, then ${k\choose 2}\ge k$.  Using the isomorphism $\beta\oplus k\epsilon\cong n\epsilon$ on $V_{n,k}$, 
we obtain\\
\[ 
\begin{array}{rcl} 
\tau V_{n,k}&\cong &k\beta\oplus {k\choose 2}\epsilon \\ 
              &\cong & \beta\oplus 
k\epsilon\oplus (k-1)\beta\oplus ({k-1\choose 2}-1)\epsilon \\
&\cong & (k-1)\beta\oplus ({k-1\choose 2}+n-1)\epsilon.\\
\end{array}   \]
A boot-strapping argument leads to the triviality of 
$\tau V_{n,k}$.   The case of the complex and quaternionic Stiefel manifolds 
can be handled in an analogues manner.  In fact,
in the case of $W_{n,2}$ and $Z_{n,2}$, boot-strapping is still possible.  Since $4\epsilon\oplus \beta\cong 2n\epsilon$ 
we have 
\[
\begin{array} {rcl} \tau W_{n,2} &=& 4\epsilon\oplus 2\beta \\ &=& 2n\epsilon\oplus\beta\\ &=& (2n-4)\epsilon\oplus 2n\epsilon\\ &=&
(4n-4)\epsilon.
\end{array}\]
The proof in the case of $Z_{n,2}$ is similar and hence omitted.
 
When $k=2$, boot-strapping fails for $V_{n,2}$.   However, it allows us to show that $\tau V_{n,k}\oplus \epsilon$ is trivial.  Thus $V_{n,2}$ is stably parallelizable.  

There does not seem to be 
any easy argument to show the parallelizability of $V_{n,2}$ although boot-strap proof is still possible  
when $n$ is even using the isomorphism $\tau \mathbb{S}^{n-1}\cong \xi\oplus \epsilon$.  
The general case requires obstruction theory. 
We refer the reader to \cite{ws} for details, where the more general case of the total space of a
sphere bundle over sphere is considered.
\end{proof}

The stable parallelizability of the Stiefel manifolds also follows from the sufficiency part of Theorem \ref{adjoint} as noted by 
Singhof \cite[p.103]{singhof}.

\subsection{The projective Stiefel manifolds}  
We begin by recalling the definition of projective Stiefel manifolds.  Although one has the 
notion of quaternionic projective Stiefel manifolds, not much is known about their span. 
(See \cite{lam}.)   For this reason we shall be contend with defining them, but 
discuss the vector field problem only for     
real and complex projective Stiefel manifolds.

The real projective Stiefel manifold $PV_{n,k}$ is defined 
as the quotient of $V_{n,k}$ under the antipodal identification: $v\sim - v$.   
Note that $PV_{n,1}$ is the real projective space  $\mathbb{R}P^{n-1}$.  The manifold $PV_{n,k}$ is the homogenous space $O(n)/(\mathbb{Z}_2\times O(n-k))$ 
where the factor $\mathbb{Z}_2$ is the subgroup $\{I_k,-I_k\}\subset O(k)\subset O(n)$.   Evidently, the 
quotient map $V_{n,k}\to PV_{n,k}$ is the double covering map which is universal except when $k=n-1$ as  
$V_{n,n-1}\cong SO(n)$. 

The complex projective Stiefel manifolds are defined similarly as $PW_{n,k}:=U(n)/(\mathbb{S}^1
\times U(n-k))$
where the factor $\mathbb{S}^1\subset U(n)$ is the 
subgroup $\{zI_k\mid |z|=1\}\subset U(k)$.  
Evidently $PW_{n,k}$ is the 
quotient of $W_{n,k}$ by the action of $\mathbb{S}^1$ where $z\cdot (w_1,\ldots, w_k)=(zw_1,\ldots, zw_k)$ and 
in fact the quotient map $W_{n,k}\to PW_{n,k}$ is the projection of a principal $\mathbb{S}^1$-bundle.

Analogously, the quaternionic projective Stiefel manifold $PZ_{n,k}$ is the homogeneous space $Sp(n)/Sp(1)\times Sp(n-k)$ 
where the factor $Sp(1)$ is subgroup $\{qI_k\mid q\in \mathbb H, ~||q||=1\}\subset Sp(k)$.  It is the quotient of $Z_{n,k}$ under the 
action of $Sp(1)$ where $q\cdot (v_1,\ldots,v_k)=(v_1\bar{q}_1,\ldots, v_k\bar{q}_k)$, $(v_1,\ldots, v_k)\in Z_{n,k}, q\in Sp(1)$. 
The quotient map $Z_{n,k}\to PZ_{n,k}$ is evidently the projection of a principal $Sp(1)$-bundle.

We denote by $\zeta_{n,k}$, or more briefly $\zeta$, the real (resp. complex) line bundle over 
$PV_{n,k}$ (resp. $PW_{n,k}$) associated to the double cover $V_{n,k}\to PV_{n,k}$ (resp. the principal $U(1)$-bundle 
$W_{n,k}\to PW_{n,k}$).   We shall denote by $\beta_{n,k} $ (more briefly $\beta$)  the bundle over $PV_{n,k}$ whose fibre over a point $[v_1,\ldots,v_k]\in PV_{n,k}$ is the orthogonal complement of $\mathbb{R}v_1+\cdots+\mathbb{R}v_k$ 
in $\mathbb{R}^n$.  The similarly defined complex vector bundle of rank $n-k$ over $PW_{n,k}$ 
will also be denoted by the same symbol $\beta_{n,k}$  (or $\beta$). 

The projection onto the $j$th coordinate $p_j:PV_{n,k}\to PV_{n,1}=\mathbb{R}P^{n-1}$ is covered by a bundle map 
of $\zeta $ on $PV_{n,k}$ and the Hopf line bundle $\xi$ on $\mathbb{R}P^{n-1}$.  Hence $p_j^*(\xi)\cong \zeta$ for $1\le j\le k$.  Using this 
one obtains the following isomorphism of real (resp. complex) vector bundles over $PV_{n,k}$ (resp. $PW_{n,k}$): 
\[k\zeta_{n,k}\oplus \beta_{n,k}\cong n \epsilon_\mathbb{F}, \eqno(7)\] 
where $\mathbb{F}=\mathbb{R}, \mathbb{C}$ as appropriate.
Equivalently, upon tensoring with $\bar \zeta$ 
and using the isomorphism $\zeta\otimes_{\mathbb F}\bar 
\zeta\cong \epsilon_{\mathbb F}$ we obtain 
\[k\epsilon_\mathbb{F} \oplus \beta_{n,k}\otimes_{\mathbb F} \bar \zeta_{n,k}\cong n\bar\zeta_{n,k}.  \eqno(8)\]
When $\mathbb F=\mathbb R$, we have 
$\bar{\zeta}\cong \zeta$.

The Hopf line bundles over the real and complex projective Stiefel manifolds have the following 
universal property.  This has been observed by S. Gitler and D. Handel \cite[p.40]{gh} and also by L. Smith \cite{smith}
for $PV_{n,k}$ where the universal property is established for real line bundles over finite complexes.  
The paper \cite{bhlsz} removed the restriction on the base space.   
We note that the formulation and proof also works for complex line bundles.   We merely state the 
result and omit its proof.

\begin{theorem} \label{universal}
Let  $\xi$ be any real (resp. complex) line bundle over a topological space $X$.   Then there exist a positive integer $n$ 
and a real (resp. complex) vector bundle $\eta$ such that $n\xi\cong \eta\oplus k\epsilon$ as real (resp. complex) 
vector bundles if and only if there exists a continuous map $f:X\to PV_{n,k} $ (resp. $X\to PW_{n,k}$) 
such that $f^*(\zeta_{n,k})\cong \xi$.  \hfill $\Box$
\end{theorem}

A description of the mod $2$ cohomology algebra was obtained by Gitler and Handel \cite{gh} which we shall 
now recall.
Let $N:=\min_{1\le j\le k}\{n-k+j\mid {n\choose n-k+j }\equiv 1\mod 2\}$.   Denote by $V=V(x_1, \ldots,x_m)$ 
a $\mathbb{Z}_2$-algebra generated by homogeneous elements $x_j, 1\le j\le m,$ such that 
$\{x_1^{\varepsilon_1}\ldots x_m^{\varepsilon_m}\mid \varepsilon_j\in \{0,1\}\}$ form a basis for the 
$\mathbb{Z}_2$ vector space $V(x_1,\ldots,x_m)$.

\begin{theorem} {\em (Gitler and Handel \cite{gh})}  With notations as above, 
the mod $2$-cohomology algebra of $PV_{n,k}$ is isomorphic to 
$\mathbb{Z}_2[y]/\langle y^N\rangle\otimes V(y_{n-k},\ldots, {y}_{N-2},y_{N}, \ldots, y_{n-1}),$ 
where $\deg(y)=1, \deg(y_j)=j, n-k\le j\le n-1, (j\ne N-1)$ 
for a suitable algebra $V$.   Furthermore, $w_1(\zeta)=y$.
\end{theorem}

Gitler and Handel also determined, almost completely, the action of the Steenrod algebra on $H^*(PV_{n,k};\mathbb{Z}_2)$.  See also \cite{bb} and \cite{antoniano}.

The following descriptions of the tangent bundle of real and complex projective Stiefel manifolds was obtained by 
Lam \cite{lam}.  
\[\tau PV_{n,k}\cong {k\choose 2}\epsilon\oplus k\zeta\otimes \beta, \eqno(9)\]
\[\tau PW_{n,k}\cong (k^2-1)\epsilon_\mathbb{R}\oplus k\bar{\zeta}\otimes _\mathbb{C}\beta,\eqno(10)\]
where we have denoted by the same symbol $\bar{\zeta}\otimes_\mathbb{C}\beta$ 
to denote its underlying real vector bundle. 

Using the isomorphism (8) one obtains the following description for the stable tangent bundle:
\[ \tau PV_{n,k}\oplus {k+1\choose 2}\epsilon \cong nk\zeta, \eqno(11)\]
\[ \tau PW_{n,k}\oplus (k^2+1) \epsilon_\mathbb{R} \cong kn\bar{\zeta}\cong nk\zeta,\eqno(12)\] 
where, again in (12), we have used $\zeta$ also to denote its underlying real vector bundle; note that 
$\bar{\zeta}\cong \zeta$ as real vector bundles. 

\begin{theorem} \label{parallel-projectivestiefel}
(i). {\em  (Zvengrowski \cite{zven-eth}, Antoniano, Gitler, Ucci, Zvengrowski \cite{aguz}.)}\\
(a) $PV_{n,k}$ is parallelizable in the following cases: 
$n=2,4,8; ~k=n-1;  ~k=2m-2, n=2m; ~ (n,k)=(16,8).$  \\
(b) $PV_{n,k}$ is not stably parallelizable in all the other cases, except possibly when $(n,k)=(12,8)$;  $PV_{12,8}$ is 
parallelizable if it is stably parallelizable.  \\
(ii) {\em (Singhof \cite{singhof}, Astey, Gitler, Micha, and Pastor \cite{agmp}.)} The complex Stiefel manifolds $PW_{n,k}$, $1\le k<n$, are not stably parallelizable except when $k=n-1.$
$PW_{n,n-1}$ is parallelizable  if $n\ge 3$; $PW_{2,1}\cong \mathbb{C}P^1=\mathbb{S}^2$ is not parallelizable.
\end{theorem}
\begin{proof}  
(i) (a) The parallelizability of $PV_{n,n-1}=V_{n,n-1}/\{\pm I\}=SO(n)/\{\pm I\}$ is implied by the Borel-Hirzebruch Theorem \ref{bh}.  

We shall now show the parallelizability of $PV_{n,n-2}$ where $n=2m$.  Let $d=\dim PV_{n,n-2}={n\choose 2}-1$.  
It is easy to see that ${n-2\choose 2} >\rho(d)$ if $n\ge 4$.  In view of the bundle isomorphism (9) and Bredon-Kosi\'nski's 
theorem, 
we see that it suffices to show that $PV_{n,n-2}$ is stably parallelizable.  
Note that $PV_{n,n-2}=SO(n)/Z\cdot SO(2)$, where $Z=\{I_{n},-I_{n}\}\subset SO(n)$ is the centre of $SO(n)$ since 
$n=2m$ is even.  Let $H=Z\cdot SO(2)\cong Z\times SO(2).$  Then the adjoint representation of $H$ is trivial since $H$ is 
abelian.  
 It follows that, in the exact sequence (4), the bundle $\nu$ is trivial. So $PV_{n,n-2}$ is stably parallelizable, as was 
 to be shown.

In the remaining cases, consider the projection $q: PV_{n,k}\to\mathbb{R}P^{n-1}$ which pulls back the 
Hopf bundle $\xi$ over $\mathbb{R}P^{n-1}$ to $\zeta$.    
By a well-known result 
of Adams \cite{adams}, the order of $\xi=\zeta_{n,1}$ is known: $2^{\varphi(n-1)}\xi=2^{\varphi(n-1)}\epsilon$ where the function $\varphi$ is defined as $\varphi(n)$ is the number of numbers $r$ such that $1\le r\le n$ such that $r\equiv 0,1,2,4 ~\mod 8$.  
We have $2^{\varphi(n-1)}=n$ if and only if $n=2,4,8$.
Since $q^*(\zeta_{n,1})=\zeta_{n,k}$, we see that $2^{\varphi(n-1)} \zeta_{n,k}\cong 2^{\varphi(n-1)}\epsilon$ for all $k<n$.
Since $\varphi(15)=7$ we have $ 2^{7}\zeta_{16,8}=2^7\epsilon$.   Therefore $\tau V_{16,8}\oplus {9\choose 2}\epsilon 
\cong 16\cdot 8\zeta\cong 2^7\epsilon$ using the isomorphism  (11).  
Also 
 $\dim V_{16,8} =120-28=92,~\rho(93)=1$ whereas span of $PV_{16,8}$ is at least  ${8\choose 2}=28$.   So by 
 the Bredon-Kosi\'nski theorem, $PV_{16,8}$ is parallelizable.  The case when $n=4,8$ 
are similarly handled and in fact easier.

(i)(b).   In several cases, fairly elementary arguments can be used to decide 
whether $PV_{n,k}$ is stably parallelizable or not.  For example, if both $n,k$ are odd, then $nk\zeta_{n,k}$ 
is not orientable.  So, by (11), we conclude that $PV_{n,k}$ is also not orientable and hence not stably parallelizable. 
However, such simple arguments leave infinitely many cases unsettled.   In \cite{aguz}, the authors 
compute the complex K-theory of $PV_{4q,k}$ which leads to determination of the (additive) order of 
$[\zeta_{4q,k}\otimes_{\mathbb R}\mathbb{C}]-[\epsilon_{\mathbb C}]\in K(PV_{4q,k}).$  
This readily leads to the determination of the order of 
$\zeta_{4q,k}$ 
up to a factor of $2$.  (By the order of a real line bundle $\xi$ we mean the smallest positive integer 
$m$ (if it exists) such that $m\xi$ is trivial; 
if the base space is a finite CW complex, it is always finite and is a power of $2$.)  
Moreover, using the inclusion $PV_{4q,k}\to PV_{4q+t,k}, 1\le t\le 3$, leads to estimation 
of the order of $\zeta_{n,k}$ for any $n$.  This is then used to show that $nk\zeta_{n,k}$ is not 
trivial for almost all the manifolds not covered in (i), still leaving out $PV_{n,k}$ where 
$(n,k)=(10,4), (12,8)$ and a few others (when $5\le n\le 7$).   When $m$ is odd, $m\xi$ is non-orientable so we may assume 
that $nk$ is even.  Thus only the cases $(7,4), (7,2), (6,3), (6,2), (5,2)$ remain, leaving out the case $(12,8)$ 
which remains at this time unresolved.  Of these, only the cases $(n,k)=(7,4), (6,3), 6,2)$ are `critical' and were proven 
to be non-stably parallelizable by a computation of the order of $\zeta_{n,k}$, using the Atiyah-Hirzebruch spectral sequence for $KO$-theory.   

In the case of $PV_{12,8}$ it was shown that  $32\zeta_{12,8}\otimes \mathbb{C}=32\epsilon_\mathbb{C}$ which 
implies that $64\zeta_{12,8}\cong 64\epsilon_\mathbb{R}$ 
but it is unknown whether $32\zeta_{12,8}\cong 32\epsilon_{\mathbb{R}}$.  
Since $\tau PV_{12,8}\oplus 36\epsilon\cong 96 \epsilon$, it remains unknown 
whether it is parallelizable or not. 

Since (9) implies that $\Span PV_{12,8} \ge 28$, 
and since $\dim PV_{12,8}= 60, ~\rho( 61)=0,$  
by the Bredon-Kosi\'nski Theorem again we see that $PV_{12,8}$ is parallelizable if it is stably parallelizable.

(ii)  Suppose that $k<n-1$.   By Singhof's theorem (Theorem \ref{torussimple}), we know that $PW_{n,k}$ is not stably 
parallelizable.  Since we did not give proof that theorem, we now proceed to give a proof of it in the special 
case of complex Stiefel manifolds.  
Using (12) we compute the Pontrjagin class $p_1(PW_{n,k})$. 
Since $\tau PW_{n,k}\otimes\mathbb{C}$ is stably equivalent to $nk\zeta_{n,k}\otimes_\mathbb{R}\mathbb{C}\cong 
nk(\zeta_{n,k}\oplus \bar{\zeta}_{n,k})$, a straightforward computation yields 
$p_1(PW_{n,k})=nkc_1(\zeta_{n,k})^2$.  Using the Gysin sequence of the principal $\mathbb{S}^1$-bundle 
$W_{n,k}\to PW_{n,k}$ and the fact that $W_{n,k}=SU(n)/SU(n-k)$ is $4$-connected when $1\le k\le n-2$, 
it is easily seen that $c_1(\zeta_{n,k})^2$ generates $H^4(PW_{n,k};\mathbb{Z})\cong \mathbb{Z}$.  Hence 
$p_1(PW_{n,k})\ne 0$ and so $PW_{n,k}$ is not stably parallelizable.      

Note that $PW_{n,n-1}=U(n)/(Z. U(1))$ is the quotient of a compact connected Lie group modulo $S=Z.U(1)
\cong \mathbb{S}^1\times \mathbb{S}^1$, which is 
a torus of rank $2$. By Lemma \ref{torus} $PW_{n,n-1}$ is parallelizable if $n>2$.  The remaining part of (ii) 
is clearly valid. 
\end{proof}

Determination of the span of a real projective manifold $PV_{n,k}$, for general values of $n,k$, is largely an open problem.  
For certain infinite set of values of $(n,k)$ the span has been determined.   
When $k$ is in the so-called 
upper range (roughly $k>n/2$)  very good estimates for 
the span of $PV_{n,k}$ have been obtained by Korba\v{s} and Zvengrowski.   (See \cite{ksz}, \cite{kz96}, \cite{kz11}.)  
It turns out that the estimates are sharp whenever span and stable span are known to be equal. 
(See Theorem \ref{stabspan}.) Usually it is easier to obtain bounds for stable span since it is possible 
to approach this using the tools of homotopy theory and K-theory.
A major source of estimates for the  
lower bound for stable span of $\Span^0(PV_{n,k})$ is the known estimate for the solution to the 
{\it generalized vector field problem}.  The generalized vector 
field problem 
asks:  {\it What is the largest value $r$ so that $m\zeta_{n,1}$ is isomorphic as a vector bundle to} 
$r\epsilon\oplus \eta$?  That is, it asks for the determination of $\Span(m\zeta_{n,1}). $  
It appears that the best known estimate for the solution to this problem general values of $m,n$ 
is due to Lam \cite{lam72}. Note that if $m<n$, then $w_m(\zeta_{n,1})\ne 0$ and so $r=0$. When $m\ge n$, we 
have $r\ge m-n+1$ since {\it any} vector bundle $\xi$ of rank $m$ over any CW complex of dimension 
$d$ is isomorphic to $(m-d)\epsilon\oplus\eta$ for a suitable vector bundle $\eta$. 
Since $\zeta_{n,k}\cong q^*(\zeta_{n,1})$ where $q$ is the projection $PV_{n,k}\to PV_{n,1}=\mathbb{R}P^{n-1}$, 
we see that $\Span (m\zeta_{n,k})\ge \Span (m\zeta_{n,1})$.   This gives us, using (11), lower bounds 
for the stable span of $PV_{n,k}$.  Combining with Theorem \ref{stabspan} which provides sufficient conditions for 
span to equal stable span results in the following.

\begin{theorem} {\em (Korba\v s, Sankaran, Zvengrowski \cite{ksz},  Korba\v s, Zvengrowski \cite{kz96}.)}  One has 
$\Span (PV_{n,k})=\Span^0 (PV_{n,k})$ in the following cases: \\
(a) $n\equiv 0\mod 2$ and $k\equiv 0,2,3,4,7\mod 8$, \\
(b) $n\equiv 1\mod 2$ and $k\equiv 0,1,4,5,6\mod 8$,\\
(c) $(n,k)=(4m, 8l+5), (4m+2, 8l+1), (4m, 16l+6), (8m, 16l+9), (8m-1,16l+7)$.\hfill $\Box$
\end{theorem}

For example, using 
(9) we saw in the course of the proof of Theorem \ref{parallel-projectivestiefel} that 
$\Span(PV_{12,8})\ge 28$.  From Koschorke's Theorem \ref{stabspan}, 
one knows that whenever 
$\chi(M)=0$ and $\dim(M)\equiv 0\mod 2$, the span of $M$ equals the stable span of $M$. 
Hence  we may use (11) to obtain 
$\Span(PV_{12,8})\ge  \Span(96\zeta_{12,8})-36 
\ge 85-36=49$.  Here we used the estimate $\Span(96\zeta_{12,8})\ge \Span(96\zeta_{12,1})\ge 96-\dim\mathbb{R}P^{11}=85$.  However, one can improve this lower bound using the work of Lam \cite{lam72} which implies that 
$\Span(96\zeta_{12,1})=91$ to obtain $\Span (PV_{12,8})\ge 55$.  See \cite{ksz} for similar estimates 
for the span of $PV_{n,k}$ for $n\le 16$. 

As for upper bounds for the stable span of $PV_{n,k}$, an obvious tool is the Stiefel-Whitney classes.  They generally work  
only for small values of $k$ and $n$ not a power of $2$.  In any case, the bounds so  
obtained are generally weak.   Another source of upper bounds uses the structure of 
$K$-ring of $PV_{n,k}$ using Theorem \ref{universal}.   Suppose that $\Span^0(\tau PV_{n,k})\ge r$.  Then 
we see that $\Span(nk\zeta_{n,k})\ge {k\choose 2}+r$.  This implies, by Theorem \ref{universal} the 
existence of a continuous map $f:PV_{n,k}\to PV_{nk,{k\choose 2}+r}$ such that $f^*(\zeta_{nk,{k\choose 2}+r})
\cong \zeta_{n,k}$.  By considering the map induced by $f$ between the $K$-rings of the spaces the following 
result was obtained. 
The structure of the ring $K(PV_{n,k})$ had been determined for $n\equiv 0\mod 4$ in \cite{aguz} and for 
all values of $n \mod 4$  in \cite{baru-hac}.

\begin{theorem} {\em (Sankaran and Zvengrowski \cite{sz5}.) } 
Let $2<k<\lfloor (n-1)/2\rfloor$.
Write $m:=\lfloor n/2\rfloor, s=\lfloor k/2\rfloor$, $d:=\dim PV_{n,k}=nk-{k+1\choose 2}$.  \\
(i) Suppose that $n\equiv 0\mod 2$.  Then $\Span^0(PV_{n,k})\le d-2q-2$ if $(-1)^q{mk-1\choose q}$ is not a 
quadratic residue modulo $2^{m-2q}$. \\
(ii) Suppose that $n\equiv 1\mod 2, k=2s$.  Then $\Span^0(PV_{n,k})\le d-2q-2$ if $(-1)^{s-q}{ns-1\choose q}$ 
is not a quadratic residue modulo $2^{m-2q}$.\\
(iii)  Suppose that $n\equiv 1\mod 2, k=2s+1$ and $1\le q<s-1, m\ge 3q$.  Then $\Span^0(PV_{n,k})\le d-2q$ if 
$(-1)^{r-q}{r\choose q}$ is not a quadratic residue modulo $2^{m-3q}$ where $r=(nk-1)/2$.
\end{theorem} 

Part (i) of the above result was stated without proof in \cite{ksz}.  

We point out here two conjectures of Korba\v{s} and Zvengrowski \cite[p. 100]{kz96}: \\
{\it Conjecture A:}  $\Span^0 (PV_{n,k})=\Span (PV_{n,k})$ for all $n,k$.\\
{\it Conjecture B:}  $\Span (PV_{n,k})\ge \kappa_{n,k}$ where $\kappa_{n,k}:=\Span (nk\zeta_{n,1})-{k+1\choose 2}$.\\

Note that Conjecture A is stronger than Conjecture B.  Indeed if conjecture A holds, then $\Span PV_{n,k}=\Span^0 (PV_{n,k})=\Span(nk\zeta_{n,k})-{k+1\choose 2}$.  Since $\zeta_{n,k}=p^*(\zeta_{n,1})$ where $p:PV_{n,k}\to PV_{n,1}=\mathbb{R}P^{n-1}$ pulls back $\zeta_{n,1} $ to $\zeta_{n,k}$, we see that $\Span(nk\zeta_{n,k})\ge \Span(nk\zeta_{n,1})$ and so 
we conclude that $\Span(PV_{n,k})\ge \kappa_{n,k}$.   Conjecture A has been verified in many cases 
by Korba\v{s} and Zvengrowski \cite{kz96} using the work of Koschorke \cite{koschorke} (Theorem \ref{stabspan} above).
They also verified Conjecture B in all cases except when $n$ is odd and $k=2$ by using a boot-strapping argument.

\subsection{Quotients of $W_{n,k}$ by cyclic groups.}
The complex Stiefel manifold $W_{n,k}=U(n)/U(n-k)$ is acted on by the circle group $\mathbb{S}^1=Z(U(n)).$ Therefore 
for any $m\ge 2$, one has a natural action of the cyclic group $\mathbb{Z}_m\subset \mathbb{S}^1$ which is 
free.
The quotient space is denoted $W_{n,k;m}$ and is called the $m$-projective (complex) Stiefel manifold. 
It is clear 
that one has a principal $\mathbb{S}^1$-bundle with projection $W_{n,k;m}\to PW_{n,k}$ and a covering projection 
$W_{n,k}\to W_{n,k;m}$ with deck transformation group 
$\mathbb{Z}_m$.  Let $\xi_{n,k;m}$ (or more briefly $\xi$) denote the complex line bundle 
which is the pull-back of the bundle $\zeta_{n,k}$ over $PW_{n,k}$.  
The smooth manifolds $W_{n,k;m}$ were studied in \cite{gs}.   We merely state here without proof the results 
obtained therein concerning 
the span and parallelizability of $W_{n,k;m}$.  We leave out the case $k=1$ which is the standard lens space
$L_m$ of dimension $2n-1$.

The tangent bundle of $W_{n,k;m}$ satisfies  the following  isomorphism of (real) vector bundles as can be seen 
using (12):
\[ \tau W_{n,k;m}\oplus k^2\epsilon=nk\xi_{n,k;m}\eqno(13).\]

We state without proof the following result due to Gondhali and Sankaran \cite{gs}.

\begin{theorem} 
Let $2\le k<n$ and $m\ge 2$.  Then\\
(i)  $\Span(W_{n,k;m})>\Span^0(PW_{n,k})\ge \dim (W_{n,k;m})-2n+1$; moreover, when $n$ is even 
$\Span (W_{n,k;m})> \dim(W_{n,k;m})-2n+3$.\\
(ii) $\Span(W_{n,k;m})>\Span^0 (W_{n,k-1;m}).$\\
(iii) $W_{n,n-1;m}$ is parallelizable.
\end{theorem}

Using Koschorke's Theorem \ref{stabspan} the following result was obtained in \cite{gs}. 

\begin{theorem}
Let $2\le k<n$ and $m\ge 2$.  Then $\Span(W_{n,k})=\Span^0(W_{n,k})$ in each of the following cases:
(i) $k$ is even, (ii)  $n$ is odd, (iii) $n\equiv 2 \mod 4$. \hfill $\Box$
\end{theorem}

Let $2\le k<n$ and $m\ge 2, 1\le r<n$.
Define positive integers $m_r$ as follows:  $m_r:=m$ if $r<n-k$; $m_r:=\gcd\{m, {n\choose j}\mid n-k<j\le r\}.$

It is easily seen that $H^2(W_{n,k;m};\mathbb{Z})\cong \mathbb{Z}_m$ generated by the first Chern class of the complex 
line bundle associated to the $\mathbb{S}^1$-extension of the universal covering projection $W_{n,k}\to W_{n,k;m}$. Denoting this generator by $y_2$, it turns out that the order of $y_2^r\in H^{2r}(W_{n,k;m};\mathbb{Z})$ is $m_r$.  
In particular, the height of $y_2$ is the largest $r\le n$ such that $m_r>1$.  By computing the Pontrjagin 
class of $W_{n,k;m}$ one obtains the following result.

\begin{theorem} Let $2\le k<n$ and let $m\ge 2$. 
With notation as above, if there exists an $r\ge 1$ such that $m_{2r}$ does not divide ${nk\choose r}$, then 
$W_{n,k;m}$ is not stably parallelizable.   In particular, if $W_{n,k;m}, k<n-1,$ is stably parallelizable, then $m$ divides 
 $nk$.  The manifold $W_{n,n-1;m}$ is parallelizable for all  $m$.
\end{theorem}

\begin{remark}{\em 
Gondhali and Subhash \cite{gondhali-subhash} introduced a generalization of complex projective Stiefel manifolds, which depend 
on a $k$-tuple $l:=(l_1,\ldots, l_k)$ of positive integers with $\gcd \{l_1,\ldots, l_k\}=1$.  These 
are homogeneous spaces $P_lW_{n,k}=U(n)/\mathbb S^1\times U(n)$ where the group $\mathbb S^1\subset U(k)$ consists of diagonal matrices $\textrm{diag}(z^{l_1},\ldots,z^{l_k}), |z|=1$.    
They obtained results 
on the (stable) parallelizability of these homogeneous spaces.  Basu and Subhash \cite{basu-subhash} obtained, among other things, upper bounds for the span of $P_l W_{n,k}$.
}
\end{remark}

\subsection{Grassmann manifolds, flag manifolds and related spaces} 

Let $\mathbb{F}G_{n,k}$ denote the space of $k$-dimensional $\mathbb{F}$-vector subspaces of 
$\mathbb{F}^n$, where $\mathbb{F}=\mathbb{R},\mathbb{C},$ or $\mathbb{H}$.   One has the following 
description of $\mathbb{F}G_{n,k}$ as a homogeneous 
space: $\mathbb{R}G_{n,k}=O(n)/(O(k)\times O(n-k))=SO(n)/S(O(k)\times O(n-k))$, $\bc G_{n,k}=U(n)/(U(k)\times U(n-k))$, and $\mathbb{H}G_{n,k}=Sp(n)/(Sp(k)\times Sp(n-k))$.   It is clear that $\mathbb{F}G_{n,1}$ is the $\mathbb{F}$-projective space $\mathbb{F}P^{n-1}$. 

More generally, one has the $\mathbb{F}$-flag manifold defined as follows:  Suppose that $\mu:=(n_1,\ldots, n_r)$ is a sequence of positive numbers with sum $|\mu|:=\sum_{1\le j\le r}n_j=:n$.  Then the real flag manifold of type $\mu$ is the coset space $O(n)/(O(n_1)\times \cdots \times O(n_r)):=\mathbb{R}G(\mu).$  The complex (resp. quaternionic) flag manifold are defined as $\mathbb{C}G(\mu)=U(n)/(U(n_1)\times \cdots\times U(n_r))$ (resp. $\mathbb{H}G(\mu):=
Sp(n)/(Sp(n_1)\times \cdots\times Sp(n_r))$) respectively.  Clearly $\mathbb{F}G(n_1,n_2)$ is just the Grassmann 
manifold 
$\mathbb{F}G_{n,n_1}$.  One may identify $\mathbb{F}G(\mu)$ with the space of flags $\underline{V}:=(V_1,\ldots, V_r)$ where 
$V_j\subset \mathbb{F}^n$ is a (left) $\mathbb{F}$-vector space of dimension $n_j$ and $V_i\perp V_j$ if $1\le i<j\le r$. 
(Note that $V_r$ is determined by the rest of the $V_j$.)  
It is clear that $\mathbb{F}G(\mu)\cong \mathbb{F}G(\lambda)$ if $\lambda$ is a permutation of $\mu$.  For this 
reason, one may assume that $n_1,\ldots, n_r$ is an increasing (or a decreasing) sequence.   It is readily verified that $\dim_\mathbb{R} \mathbb{F}
G(\mu)=(\dim_\mathbb{R}\mathbb{F})(\sum_{1\le i<j\le r} n_in_j)$.   It turns out that any complex flag manifold 
has the structure of a complex projective variety.    When $r>2$, one has an obvious fibre bundle projection 
$p_j:\mathbb{F}G(\mu)\to \mathbb{F}G(\mu(j)), 1\le j<r,$ where $\mu(j)$ is the sequence obtained from $\mu$ by 
replacing $n_j, n_{j+1}$ by $n_j+n_{j+1}.$   The fibre of this bundle is readily seen to be the Grassmann 
manifold $\mathbb{F}G(n_j,n_{j+1})$.  

The complex and quaternionic flag manifolds are simply connected.  However, this is not true of 
the real flag manifolds. Indeed, $\pi_1(\mathbb{R}G(\mu))\cong (\mathbb{Z}_2)^{r-1}$ 
(where $r$ is the length of $\mu$), except when $n=2, \mu=(1,1)$ which corresponds to the case 
of the circle, $\mathbb{R}P^1$.   The {\it oriented flag manifold} of type $\mu$, denoted $\wt{G}(\mu)$ 
is defined as the coset space 
$SO(n)/(SO(n_1)\times \cdots\times SO(n_r))$.  It may be identified with the space of all oriented 
flags $(V;\sigma), V\in \mathbb{F}G(\mu)$ and $\sigma=(\sigma_1,\ldots,\sigma_r)$ where $\sigma_j$ 
is an orientation on $V_j, 1\le j\le r$ with the restriction that these orientations induce the 
standard orientation on $\mathbb{R}^n=V_1\oplus\cdots\oplus V_r$.  (Thus $\sigma_r$ is determined by $\sigma_j, 1\le j<r$.) The natural projection 
$q:\wt{G}(\mu)\to \mathbb{R}G(\mu), (V,\sigma)\mapsto V$ which forgets the orientations on the 
flags is a covering map. It is universal except when 
$\mu=(1,1)$, in which case it is the double covering of the circle.    The deck transformation group is generated by 
the elements $t_j: \wt{G}(\mu)\to \wt{G}(\mu), 1\le j<r,$ which reverses the orientation on $V_j$ 
and on $V_r$.   

Let $\gamma_j(\mu)$ (more briefly $\gamma_j$) denote the canonical $n_j$-plane bundle over $\mathbb{F}G(\mu)$ whose 
fibre over a flag $\underline{V}=(V_1,\ldots, V_r)\in \mathbb{F}G(\mu)$ is the vector space $V_j$. Evidently 
we have the $\mathbb{F}$-bundle isomorphism
\[\gamma_1(\mu)\oplus\cdots\oplus \gamma_r(\mu)\cong n\epsilon_\mathbb{F}.\eqno(14)\]
The tangent bundle of $\mathbb{F}G(\mu)$ has the following description:
Recall that $\hom_\mathbb{F}(\xi, \eta)\cong \bar{\xi}\otimes \eta$ 
as $Z(\mathbb{F})$-vector bundles.  Here $\bar{\xi}$ denotes the same underlying real vector bundle 
but with conjugate $\mathbb{F}$-structure;  when $\mathbb{F}=\mathbb{R}$, $\bar{\xi}=\xi$.  
\[\tau \mathbb{F}G(\mu)\cong \oplus_{1\le i<j\le r}\hom_{\mathbb{F}}(\gamma_i,\gamma_j)\cong\oplus_{1\le i<j\le r}\bar{\gamma}_i\otimes_{\mathbb F} \gamma_j\eqno(15) \]
as $Z(\mathbb{F})$-vector bundles where $Z(\mathbb{F})$ denotes the centre of the division ring $\mathbb{F}$.
We refer the reader to \cite{lam} for a proof.  

We shall denote the bundle $q^*(\gamma_j)$ by $\wt{\gamma}_j$. Then, from (15), we see that tangent bundle $\tau \wt{G}(\mu)$ is isomorphic to $\oplus_{1\le i<j\le r} \wt{\gamma}_i\otimes \wt{\gamma}_j$.   The bundle $\wt{\gamma}_j$ is canonically oriented:  the orientation on the fibre of $\wt{\gamma}_j$ over $(\underline{V}, \sigma)\in \wt{G}(\mu)$ is the {\it 
oriented vector space} $(V_j,\sigma_j)$; it follows that the tangent bundle of $\wt{G}(\mu)$ is also canonically 
oriented.  The deck transformation $t_j$ induces a bundle isomorphism $Tt_j:\tau \wt{G}(\mu)\to \tau \wt{G}(\mu)$ which 
preserves the summands $\wt{\gamma}_k\otimes\wt{\gamma_l}.$  It preserves the orientation  on $\wt{\gamma}_k\otimes 
\wt{\gamma}_l$, $k<l$,  if and only if one of the following holds: 
(a) $(k,l)=(j,r), n_j\equiv n_r \mod 2$, (b) 
$k=j, l<r, n_l$ is even, (c) $k\ne j, l=r, n_k$ is even, or (d) $\{k,l\}\cap \{j,r\}=\emptyset$.   As $t_j$ preserves the orientation 
on $\tau \wt{G}(\mu)$ if and only if it reverses the orientation on an even number of summands $\wt{\gamma}_{k}\otimes 
\wt{\gamma}_{l}, 1\le k<l\le r$, it follows that $t_j$ is orientation preserving on $\wt{G}(\mu)$ if and only if 
$n_j\equiv n_r \mod 2$.   Hence   
it follows that $\mathbb{R}G(\mu)$ is orientable if and only if $n_j\equiv n_r $ for {\it every} $j, 1\le j<r.$  
This fact may also be verified by computing the first Stiefel-Whitney class.
As remarked already, the complex and quaternionic flag manifolds are simply connected. It follows that 
they are orientable.  

We have the following theorem concerning the (stable) parallelizability of $\mathbb{F}$-flag manifolds.  
The case of Grassmann manifolds $\mathbb{F}G(n_1,n_2)$ was settled in full generality by Trew and Zvengrowski \cite{tz}.  
See also \cite{hs}, \cite{y}, \cite{bako}.  The (stable) parallelizability of $\mathbb F$-flag manifolds 
was settled by Sankaran and Zvengrowski \cite{sz}.  It turns out that the proof in 
most of the cases $r\ge 3$ follows easily from the results on Grassmann manifolds.  The result for the class of complex flag manifolds is a special case of a more general result, namely Theorem \ref{unitarygroupquotients}, due to Singhof \cite{singhof}.   
See also \cite{sw}, \cite{sw-err} where the result for quaternionic flag manifolds was  obtained.  Korba\v{s} \cite{korbas85} obtained 
the results for real flag manifolds using Stiefel-Whitney classes.

\begin{theorem}  Let $\mu=(n_1,\ldots, n_r)$ where $n_1\ge \ldots \ge n_r\ge 1$, $r\ge 2,$ and let $n:=\sum_{1\le j\le r}n_j.$ Let $\mathbb{F}=\mathbb{R},\mathbb{C}, 
\mathbb{H}$.  Then $\mathbb{F}G(\mu)$ is stably parallelizable in the following cases:\\
(i) $\mathbb{F}=\mathbb{R}, \mathbb{C},\mathbb{H}$ and  $n_1=1$ for all $j$.  Moreover, when $n_1=1$,  
$\mathbb FG(\mu)$ is parallelizable only when $\mathbb F=\mathbb R$.   When $\mathbb F=\mathbb C,\mathbb H$, the remaining 
$\mathbb F$-flag manifolds are not stably parallelizable.\\ 
(ii) Let $\mathbb F=\mathbb R$ and $n_1>1, r\ge 2$.  Other than the projective spaces $\mathbb{R}P^{n_1}=\mathbb{R}G(n_1,1)$ when $n_1=3, 7$, 
none of the flag manifolds $\mathbb RG(\mu)$ are stably parallelizable.
\end{theorem}

\begin{proof}
We shall first consider the case $r=2$ namely that of the Grassmann manifold.    
Since $\mathbb{F}G_{n,k}\cong \mathbb{F}G_{n,n-k},$ we may assume that  $1\le k\le n/2$, $n\ge 4$.  Since the case $k=1$ 
corresponds to the $n-1$-dimensional projective space which is well-known and classical, we shall assume 
that $k\ge 2$.  As is customary, we shall denote the canonical bundles $\gamma_1, \gamma_2=\gamma_1^\perp$ over 
$\mathbb{F}G_{n,k}$ by $\gamma_{n,k},\beta_{n,k}$ respectively.

One has an inclusion $h_j:\mathbb{F}G_{n-j,k-j}\subset \mathbb{F}G_{n,k},1<j<k$, induced by the inclusion of 
$\mathbb{F}^{n-j}\subset \mathbb{F}^n$. Explicitly, $h_j(V)=V+\mathbb{F}e_{n-j+1}+\cdots+\mathbb{F}e_n\in 
\mathbb{F}G_{n,k}, ~\forall ~V\in 
\mathbb{F}G_{n-j,k-j}$.  (Here $e_i$ denotes the standard basis element.)   Now $h_j^*(\gamma_{n,k})=\gamma_{n-j,k-j}
\oplus j\epsilon_\mathbb{F}, h_j^*(\beta_{n,k})=\beta_{n-j,k-j}$.   Therefore, we have the following $Z(\mathbb{F})$-bundle isomorphisms:
\[
\begin{array}{rcl} h_j^*(\tau \mathbb{F}G_{n,k})&\cong & h_j^*(\bar{\gamma}_{n,k})\otimes_\mathbb{F}h_j^*(\beta_{n,k})\\ 
&=& (\bar{\gamma}_{n-j,k-j}\oplus j\epsilon_\mathbb{F})\otimes \beta_{n-j,k-j} \\ 
&=& \tau \mathbb{F}G_{n-j,k-j}\oplus j\beta_{n-j,k-j}.
\end{array}\]
Put $j=k-1$ and use the isomorphism $\tau \mathbb{F}G_{n,1}\oplus \epsilon_\mathbb{F} \cong n\bar{\gamma}_{n,1}$ 
to obtain 
$h_{k-1}^*(\tau \mathbb{F}G_{n,k})\oplus \epsilon_\mathbb{F}=(n-k+1)\bar{\gamma}_{n-k+1,1}\oplus (k-1)\beta_{n-k+1,1}.$
Now use the isomorphism $\gamma_{n-k+1,1}\oplus\beta_{n-k+1,1}\cong (n-k+1)\epsilon_\mathbb{F}$ and the 
fact that, for any $\mathbb{F}$-vector bundle $\xi$, we have the isomorphism  $\bar{\xi}\cong \xi$ of real vector bundles,
we obtain, in $KO(\mathbb{F}G_{n-k+1,1})$ the following:
\[
\begin{array}{rcl} h^*([\tau \mathbb{F}G_{n,k}]) &=&(n-k+1)[\gamma_{n-k+1,1}]+(k-1)[\beta_{n-k+1,1}]\\
&=&(n-k+1)[\gamma_{n-k+1,1}+(k-1)
\epsilon_\mathbb{F}]-(k-1)[\gamma_{n-k+1}]\\
&=& (n-2k+2)[\gamma_{n-k+1,1}]+d(n-k+1)(k-1)\epsilon_{\mathbb R}\\
\end{array}\]
where $d=\dim_\mathbb{R}\mathbb{F}$.

When $\mathbb{F}=\mathbb{R}$, it follows from the known order of the Hopf bundle $\gamma_{n-k+1,1}$, that 
$(n-2k+2)[\gamma_{n-k+1,1}]\ne 0$ in $KO(\mathbb{R}P^{n-k})$ since $(n-k)\ge k\ge 2$.  When 
$\mathbb{F}=\mathbb{C}, \mathbb{H}$,     
an easy computation of the first Pontrjagin class (of the underlying real vector bundle) shows that $(n-2k+2)\gamma_{n-k+1,1}$ is not stably trivial as a real vector bundle.    (Trew and Zvengrowski \cite{tz} 
altogether avoided computation of Pontrjagin classes, but used information about $KO$-theory in the case of complex and 
quaternionic projective spaces.)
Thus in all cases,  $\mathbb{F}G_{n,k}$ is not stably parallelizable. 

It remains to consider the case $r\ge 3$.  In this case, we have a fibre inclusion  
$\mathbb{F}G(n_1,n_2)\hookrightarrow\mathbb{F}G(\mu)$ of the fibre bundle projection 
$\mathbb{F}G(\mu)\to \mathbb{F}G(n_1+n_2,\ldots, n_r)$.  
Since the normal bundle to fibre inclusion is trivial, we see that the tangent bundle of $\mathbb{F}G(\mu)$ 
restricts to the stable tangent bundle of $\mathbb{F}G(n_1,n_2)=\mathbb{F}G_{n_1+n_2,n_2}$.  Therefore 
$\mathbb{F}G(\mu)$ is not stably parallelizable except, possibly, when $\mathbb{F}G(n_1,n_2)$ is stably 
parallelizable.  Thus $\mathbb{F}G(\mu)$ is not stably parallelizable, except possibly in the following cases: 
(i) $\mathbb{F}=\mathbb R, \mathbb{C}, \mathbb{H},$ and $n_1=1$; (b) $\mathbb{F}=\mathbb{R}$ and $n_2=1, n_1=3,7$.

Case (i):  The parallelizability of $\mathbb{R}G(1,\ldots, 1)$ follows from Therem \ref{bh}.  
Note that $\mathbb{C}G(1,\ldots, 1)=U(n)/(U(1)\times \cdots\times U(1))$ is stably parallelizable by 
Theorem \ref{torus}.  The stable parallelizability of $\mathbb HG(1,\ldots, 1)$ was first proved by Lam \cite{lam}, 
making essential use of the functor $\mu^2$, which is an analogue of the second exterior power in the real and complex case.   

Case (ii):    
We need only show that $\mathbb{R}G(3,1,1)$ and $\mathbb{R}G(7,1,1)$ are not 
stably parallelizable.   Note that one has a double covering projection $PV_{n,2}\to \mathbb{R}G(n-2,1,1)$ defined 
as $[v_1, v_2]\mapsto (\{v_1,v_2\}^\perp, \mathbb{R}v_1, \mathbb{R}v_2)$.  Since $PV_{5,2}$ and $PV_{9,2}$ are 
not stably parallelizable by Theorem \ref{parallel-projectivestiefel}, it follows that $\mathbb{R}G(3,1,1), \mathbb{R}G(7,1,1)$ are also not stably parallelizable.  
\end{proof}

We shall write $G(\mu)$ to denote the real flag manifold of type $\mu$.  We assume that $n_1\ge \cdots\ge n_r$.
We observed already that the map $q: \wt{G}(\mu)\to G(\mu)$ that forgets the orientations on the flags is a 
covering projection with deck transformation group $(\mathbb{Z}_2)^{r-1}$ generated by $t_j, 1\le j<r$.  
Recall that $\wt{\gamma}_i(\mu)$ or more briefly $\wt{\gamma}_i$ is the pull-back bundle $q^*(\gamma_i(\mu))$.   
In the case of the oriented Grassmann manifolds, we have $r=2$ and we write $\wt{\gamma}_{n,k}, \wt{\beta}_{n,k}$ to 
denote $\wt{\gamma}_1, \wt{\gamma}_2$ respectively.   
From (14) we obtain that $\wt{\tau}(\mu):=\tau(\wt{G}(\mu))=\bigoplus_{1\le i<j\le r}\wt{\gamma}_i\otimes \wt{\gamma}_j$.

Next we have the following result concerning the oriented flag manifolds.
As usual, $\mu=(n_1,\ldots,n_r)$, $n=\sum_{1\le j\le r}n_j$.  

\begin{theorem} {\em (\cite{mm}, \cite{sz2}, \cite{sz4}.) }
(i) Let $2\le k\le n/2$.  The oriented Grassmann manifold $\wt{G}_{n,k}$ is stably parallelizable if and only if $(n,k)=(4,2),(6,3)$.   The manifold $\wt{G}_{6,3}$ is parallelizable but $\wt{G}_{4,2}\cong \mathbb{S}^2\times \mathbb{S}^2$ is not. \\
(ii) Let $r\ge 3$.   Then $\wt{G}(\mu)$ is stably parallelizable if and only if any one of the following holds: $\{n_1,\ldots, n_r\}$ is contained in $\{1,2\}$ or in $\{1,3\}$.
\end{theorem}
 
\begin{proof}  
Consider the inclusion $\wt G_{n-2,k-1}\stackrel{i}{\hookrightarrow} \wt G_{n-1,k-1}\stackrel{j}{\hookrightarrow} \wt G_{n,k}$ 
induced by the inclusion of $\mathbb R^{n-2}\hookrightarrow \mathbb R^{n-1}\hookrightarrow \mathbb R^n$.  Explicitly 
$i(V)=V$ and $j(U)=U+\mathbb Re_n$ with the direct sum orientation, got by adjoining $e_n$ to an oriented basis of $U$.
Denote by $q$ the composition $j\circ i$.  It is easily verified that the normal bundle to the embeddings $i, j, q$ are respectively 
$\wt\beta_{n-2,k-1}, \wt \gamma_{n-1,k-1}, (n-2k-2)\epsilon$ respectively.  Therefore the composition of the embeddings of 
type $q$ leads to embeddings  
$f:\wt G_{n-2k+4,2}\stackrel{q}{\to} \wt G_{n,k}$ when $ n>2k,$ and 
$g:\wt G_{8,4}\to \wt G_{n,k}$ when $n=2k\ge 8$, with  trivial normal bundle in each case.   

So we need only show that $\tau \wt G_{m,2}, m\ge 5,$ and $\tau \wt G_{8,4}$ are not stably trivial. It turns out 
$w_2(\wt G_{5,2})\ne 0$. 

Let $m\ge 6$.  
Using the embedding $i$ repeatedly, we obtain an embedding $\wt G_{6,2}\to \wt G_{m,2}$ we get $\tau \wt{G}_{m,2}
=\tau \wt G_{6,2}\oplus (m-6)\wt \gamma_{6,2}$. 
Similarly, under the embedding $\wt G_{6,2}\hookrightarrow \wt G_{7,3} \hookrightarrow \wt G_{8,4}$, $\tau \wt G_{8,4}$ restricts 
to $\tau \wt G_{6,2}\oplus 2\wt \beta_{6,2}=\wt \gamma_{6,2}\otimes \wt\beta_{6,2}\oplus 2\wt\beta_{6,2} $. 

Working in the ring $KO(\wt G_{6,2})$ we have $[\wt\beta_{6,2}]=6-[\gamma_{6,2}]$ and 
we must show that the elements $x_m:=[\wt\gamma_{6,2}](6-[\wt\gamma_{6,2}])+(m-6)[\wt\gamma_{6,2}]-2(m-2)$ and 
$y:=[\wt\gamma_{6,2}]\cdot (6-[\wt\gamma_{6,2}]+2(6-\wt\gamma_{6,2}]-16$ are not zero.  Set $z=[\wt\gamma_{6,2}]-2$. 
We have $x_m:=(z+2)(4-z)+(m-6)(z+2)-2(m-2)=-z^2+(m-4)z$ and $y=(z+2)(4-z)+2(4-z)-16=-z^2$.  

Now consider the map $h:\mathbb CP^2\to \wt G_{6,2} $ obtained by regarding any complex line $L\subset \mathbb C^3
\cong \mathbb R^{6}$ as a real vector space of dimension $2$ with its natural orientation.  Then 
$h^*(\wt \gamma_{6,2})=\mathbb C\gamma_{3,1}=:\zeta.$   The ring homomorphism $h^*:KO(\wt{G}_{6,2})\to 
KO(\mathbb CP^2)$ maps $z$ to $[\zeta]-2$.  Since $[\zeta]-2, ([\zeta]-2)^2$ generate a free abelian group of rank $2$ 
in $KO^*(\mathbb CP^2) $ by \cite{fujii}, it follows that the same is true of $z,z^2\in KO(\wt G_{6,2})$ and so we conclude 
that $x_m\ne 0~\forall m\ge 6$, and $y\ne 0$.   

The stable parallelizability of $\wt G_{4,2}=SO(4)/SO(2)\times SO(2)$ is immediate from Proposition \ref{torus}.  By the same 
proposition, if $n_1\le 2$, then $\wt {G}(\mu)=SO(n)/S$ where $S$ is a torus and hence it is stably parallelizable. 
It is parallelizable precisely if $S$ is not a maximal torus. 

Next suppose that $n_i=3, 1\le i\le s,$ and $n_j=1$ for $s<j\le r$ for some $s$.     For any oriented vector bundle 
$\xi$ of rank $m$, we have $\Lambda^p(\xi)\cong \Lambda^{m-p}(\xi)$.  Hence we have $\Lambda^2(\wt\gamma_{i})\cong 
\wt\gamma_i, 1\le i\le s$.  Also $\Lambda^2(\wt\gamma_j)=0$ for $j>s$.  
Applying $\Lambda^2$ to both sides of the isomorphism,  
isomorphism $\oplus_{1\le j\le r}\wt \gamma_j=n\epsilon$, and using (15) we obtain 
$\tau \wt G(\mu)=\oplus_{1\le i<j\le r} \wt \gamma_i\otimes \wt \gamma_j$, 
we obtain ${n\choose 2}\epsilon=\tau \wt G(\mu)\oplus (\oplus_{1\le j\le s}\Lambda ^2(\wt \gamma_j))=\tau \wt G(\mu)
\oplus(\oplus _{1\le i\le s} \wt {\gamma}_i)$.  Since $\wt \gamma_j\cong \epsilon$ as $n_j=1$ for $j>s$, we obtain 
$({n\choose 2} +(r-s))\epsilon=\tau \wt G(\mu)\oplus (\oplus_{1\le j\le r}\wt \gamma_j)=\tau \wt G(\mu)
\oplus n\epsilon$ and so $\wt G(\mu)$ is stably parallelizable.   
As for the parallelizability of $\wt G(\mu)$, we apply the Bredon-Kosi\'nski Theorem \ref{b-k}. Evidently 
$\chi(\wt G(\mu))=0$, implying the parallelizability when the dimension is even.  When the dimension is 
odd, leaving out the case $\wt G(3,1)=\mathbb S^3$ and the $7$-dimensional manifold $\wt{G}(3,1,1)$ which 
are parallelizable, 
we must show that the mod  2 Kervaire semicharacteristic $\hat\chi_2(\wt G(\mu))$ vanishes.  This follows from 
the fact that $\wt G(\mu)$ admits a fixed point free $\mathbb Z_2\times \mathbb Z_2$-action (see 
Remark \ref{lmp}).  

Finally, suppose that there exists $i, j\le r$ such that $n_i\ge 3, n_j\ge 2$ but $n_j\ne 3$, then 
the oriented Grassmann manifold $\wt G(n_i,n_j)$ is not stably parallelizable.  Since $\wt G(\mu)$ is 
fibred by $\wt G(n_i,n_j)$, it follows that $\wt G(\mu)$ is not stably parallelizable.
\end{proof}

\begin{remark}{\em 
One has the universal double cover $Spin(4)\to SO(4)$ under which the maximal torus $SO(2)\times SO(2)$ lifts to 
a maximal torus $T$.  One has an isomorphism of Lie groups $Spin(4)\cong Spin(3)\times Spin(3)$ under which $T$ corresponds 
to $\wt{T}=Spin(2)\times Spin(2)$. So 
\[
\begin{array} {rcl} \wt{G}_{4,2}&=& SO(4)/(SO(2)\times SO(2))\\ &\cong &Spin(4)/\wt{T}\\ & \cong & Spin(3)/Spin(2)\times 
Spin(3)/Spin(2) \\& =& \mathbb S^2\times \mathbb S^2.\\
\end{array}\]   
}
\end{remark}

Note that span of complex and quaternionic Grassmann 
manifolds are zero since they have non-vanishing Euler-Poincar\'e characteristic; see (2). 
We have the following 
general result concerning the span of real Grassmann manifolds.   This is essentially due to  
Leite and Miatello \cite{lm} who considered 
oriented Grassmann manifolds. The proof 
given here is due to Zvengrowski (unpublished).  

\begin{theorem} \label{span-grassmann}
When $k$ is even or $n$ is odd, $\Span (\br G_{n,k})=0$.
When $k$ is odd and $n$ even, \[\Span (\br G_{n,k})\geq \Span(\bs^{n-1})=\rho(n)-1.\eqno(16)\]
\end{theorem}
\begin{proof}  The rank of $G:=SO(n)$ equals $\lfloor n/2\rfloor$.  Let $H=S(O(k)\times O(n-k))$.  Then $H_0=
SO(k)\times SO(n-k)$ has the same rank 
as $SO(n)$ if and only if $n$ is odd or $k$ is even.  Since $\mathbb RG_{n,k}=G/H$, it follows from 
Theorem \ref{spanvanishing} that $\Span(G_{n,k})>0$ if and only if $n$ is even and $k$ is 
odd. 

Let  $n$ be even and $k$ odd.  Let $r=\rho(n)$ and let $\mu_1=id, \mu_2,\ldots, \mu_r:\mathbb{R}^n\to\mathbb{R}^n$ be 
the Radon-Hurwitz transformations.
Thus $\mu_i\mu_j=-\mu_j\mu_i, \mu_j^2=-id, 2\le i< j\le r$.   (See \S \ref{spanofspheres}.)

Let $v_j(X)\in \hom(X,X^\perp)$ be the composition $X\hookrightarrow \mathbb{R}^n\stackrel{\mu_j}{\to}\mathbb{R}^n  
\stackrel{p}{\to} X^\perp$, $2\le j\le r$.  Here $p$ denotes the orthogonal projection.  
Then $v_j(X)\in T_X\mathbb{R}G_{n,k}, 2\le j\le r,$ and we obtain smooth vector fields $v_2,\ldots, v_r$ on 
$\mathbb{R}G_{n,k}$.  We claim that 
these are everywhere linearly independent on $\mathbb{R}G_{n,k}$.  To see this, note that if $(a_2, \ldots, a_r)\in \mathbb{S}^{r-1}$,  then $\mu:=\sum_{2\le j\le r}a_j\mu_j$ is a skew-symmetric orthogonal transformation of $\mathbb{R}^n$.  
Hence it does not have an {\it odd-dimensional} invariant subspace in $\mathbb{R}^n$.  It follows that, for any $X\in \mathbb{R}G_{n,k}$,  the composition 
$X\hookrightarrow \mathbb{R}^n\stackrel{\mu}{\to}\mathbb{R}^n\stackrel{p}{\to}X^\perp$ is non-zero.  Therefore 
$\sum_{2\le j\le r}a_jv_j(X)\ne 0.$   Thus $\Span(\mathbb{R}G_{n,k})\ge r-1=\rho(n)-1.$
\end{proof}

\begin{remark} \label{spanequalsrh}
{\em  (i) Write $n=2m+a, a=0,1$.
The Weyl group $W(SO(n),T)$ is isomorphic to the semidirect product $\mathbb Z_2^r\rtimes S_m$ where $r=\lfloor (n-1)/2\rfloor.$ The symmetric group $S_m$ acts on $\mathbb Z_2^m$ by permuting the 
factors. When $n=2m$, the group $\mathbb Z_2^{r}$ is identified with the subgroup $\{u=(u_1,\cdots, u_m)\in 
\mathbb Z_2^m\mid \sum u_j=0\}.$  Thus $|W(SO(n),T)|=2^rm!$.  See, for example, \cite[Ch. 14, \S 7]{husemoller}.
Using formula (2), we see that $\chi(\wt{G}_{n,k})
=2 {\lfloor n/2\rfloor\choose \lfloor k/2\rfloor}$ when $n$ is odd or $k$ is even.
The following formula can be obtained from Equation (2). When $n$ is odd or $k$ is 
even, $\chi(\br G_{n,k})= {\lfloor n/2\rfloor\choose \lfloor k/2\rfloor}$ and so $\Span( \br G_{n,k})$ equals zero.  When $n$ is even 
and $k$ is odd, $\chi(\wt{G}_{n,k})=0=\chi(\mathbb RG_{n,k})$ since rank $SO(n)$ exceeds the rank of $SO(k)\times SO(n-k)$. 
The same argument shows that $\chi(\wt{G}(\mu))=0=\chi(G(\mu))$ if and only if there exist $i<j\le r$ 
such that both $n_i,n_j$ are odd.   Thus $\Span(G(\mu))=\Span(\wt{G}(\mu))=0$ if and only if at most one of the numbers 
$n_1,\ldots, n_r$ is odd.  Also 
Theorem \ref{span-grassmann} implies that $\Span(G(\mu))\ge \rho(n)-1$.  In general, however, very 
little information is available concerning the span of a general real flag manifold. 
See \cite{korbas86} and \cite{ilori-ajayi} for span of $\mathbb{R}G(n-2,1,1)$ for special values of $n$. 

(ii) It is known that 
equality holds in (16) in infinitely many cases.   We point out a sample of such results obtained in 
\cite{sankaran-thesis}. 
For example, $\Span(\mathbb RG_{n,3})=3$ when $n$ is of the form $4(2^r+1)$. 
This follows by showing that $w_{d-3}(\mathbb RG_{n,k})\ne 0$ where $d=k(n-k)=\dim \mathbb RG_{n,k}$.  Also when $n\equiv 2\mod 4, k\equiv 1\mod 2$, we have $\dim \mathbb RG_{n,k}\equiv 1\mod 4$.  
Since $n$ is even and $k$ is odd, $\mathbb RG_{n,k}$ is orientable and is an unoriented boundary (see \cite{sankaran-cmb}). 
Using Remark \ref{lmp} it can be seen that the Kervaire semicharacteristic $\kappa(\mathbb RG_{n,k})$ equals $0$ or $1$ according as ${n\choose k}\equiv 0$ or $2 \mod 4$.  
So it follows from Theorem \ref{span2} that $\Span(\mathbb RG_{n,k})=1$ (resp. $\Span(\mathbb RG_{n,k})\ge 2)$ if 
${n\choose k}\equiv 2\mod 4$ (resp. ${n\choose 2}\equiv 0\mod 4$).  
}
\end{remark} 

The last remark should be compared with the following result.   

\begin{theorem} {\em (Leite and Miatello \cite{lm}.)} Let $n-k=2r+1, k=2s+1, $ where $s\ge 1,$ and $r\ge 1$ is odd.  
Suppose that $r+s$ is even so that $n\equiv 2\mod 4$.
Then:\\ $\Span(\wt{G}_{n,k})=\begin{cases}
                                            1&~\textrm{if ~} {r+s\choose r}\equiv 1\mod 2,\\
                                              \ge 2&~\textrm{if~} {r+s\choose r}\equiv 0\mod 2.\\
                                              \end{cases} $
\hfill $\Box$
\end{theorem}

In general, the determination of span of real Grassmann manifolds is a wide open 
problem.   The first `non-trivial' case is that of the Grassmann manifold $\mathbb{R}G_{6,3}$.  
In this case the Radon-Hurwitz bound yields  $\Span(\mathbb{R}G_{6,3})\ge 1$. 
We have the following result:

\begin{theorem}
{\em (Korba\v{s}-Sankaran \cite{ks}.)} $\Span (\mathbb{R}G_{6,3})=7$.
\end{theorem}
\begin{proof}
Recall that $\tau(G_{n,k})\cong \gamma_{6,3}\otimes \beta_{6,3}.$ Since $\gamma_{6,3}\oplus \beta_{6,3}=6\epsilon$, 
taking the second exterior power on both sides we obtain $\Lambda^2(\gamma_{6,3}\oplus \beta_{6,3})=15\epsilon$.  Expanding the left hand side we obtain $\Lambda^2(\gamma_{6,3}\oplus\beta_{6,3})
=\Lambda^2(\gamma_{6,3})\oplus \gamma_{6,3}\otimes \beta_{6,3}\oplus\Lambda^2(\beta_{6,3})
=\gamma_{6,3}\otimes \xi\oplus \tau\mathbb{R}G_{6,3}\oplus\beta_{6,3}\otimes \eta$, 
where in the last equality $\xi:=\det(\gamma_{6,3}),\eta:=\det(\beta_{6,3})$ and we have used  the bundle isomorphism 
$\Lambda^r(\omega)\cong \Lambda^{m-r}\omega\otimes \det(\omega)$ for any vector bundle $\omega$ of rank $m$. 
Since a real line bundle is determined by its first Stiefel-Whitney class, 
 it is readily seen that 
 $\xi\cong \eta$ and so we obtain 
 \[15\epsilon=\tau\mathbb{R}G_{6,3}\oplus (\gamma_{6,3}\oplus \beta_{6,3})\otimes \xi
 =\tau \mathbb{R}G_{6,3}\oplus 6\xi.\eqno(17)\]  
 
 It is known that there exists a $3$-fold {\it vector product} $\mu:\mathbb{R}^8\times \mathbb{R}^8\times \mathbb{R}^8\to \mathbb{R}^8$; see \cite{wh}.  An explicit formula was given by Zvengrowski \cite{z}.   
 The map $\mu$ has the following properties: (a) $\mu$ is 
 multilinear, (b) $\mu(u,v,w)=0$ if $u,v,w$ are linearly dependent, (c) $\mu(u,v,w)\in \mathbb{R}^8$ is a unit 
 vector that  
 depends only on the oriented $3$-dimensional vector space spanned by $u,v,w$ if they are pairwise 
 orthogonal.  
The map $V\mapsto \mathbb{R}\mu(u,v,w)$ is a well-defined continuous map 
 $f:\mathbb{R}G_{8,3}\to \mathbb{R}P^7$ where $u,v,w$ is any orthonormal basis $V\in \mathbb{R}G_{6,3}$.
 It is not difficult to see that $f$ induces an isomorphism of fundamental groups.  From this it follows easily that 
 $\det(\gamma_{8,3})\cong f^*(\gamma_{8,1})$.  Restricting 
 to $\mathbb{R}G_{6,3}$ and using the fact that $8\gamma_{8,1}\cong 8\epsilon$, we obtain that $8\xi\cong 8\epsilon$ 
 whence $15\epsilon=7\epsilon\oplus 8\xi$.  Therefore, using (17) we conclude that $\tau\mathbb{R}G_{6,3}$ is 
 {\it stably isomorphic} to $7\epsilon \oplus 2\xi$.   Hence $\Span^0(\mathbb{R}G_{6,3})\ge 7$.  
 By a straightforward computation we have $w_2(\mathbb RG_{6,3})\ne 0$ and so $\Span^0(\mathbb RG_{6,3})=7$.    
 
Since by Remark \ref{spanequalsrh}(ii), $R_L(\mathbb RG_{6,3})=
 \kappa(\mathbb RG_{6,3})=0$, by appealing to Theorem \ref{stabspan} we conclude that 
 $\Span(\mathbb RG_{6,3})=\Span^0(\mathbb RG_{6,3})=7$.  
 \end{proof} 

For results on the (stable) parallelizability of {\it partially oriented flag manifolds} the reader is referred to \cite{sz2}, \cite{sz3}.  
For the orientability of a generalized real flag manifolds see \cite{patrao}.
  
\section{Homogeneous spaces for non-compact Lie groups}
We now turn to the case where $G$ is a connected {\it non-compact} Lie group. 

For convenience we will assume that $G$ is a connected {\it linear} Lie group, that is, 
$G$ is a closed connected subgroup of $GL(N,\mathbb R)$ for some $N$.
Let $R=\textrm{rad}(G)$ be the {\it radical} of $G$, i.e., the maximal connected normal solvable subgroup of $G$.  
Then $\bar{G}:=G/R$ is semisimple and $R$ is a semidirect product $R_u\rtimes T$ where $R_u$ is the {\it unipotent radical} of 
$G$, namely, the maximal connected normal nilpotent subgroup of $G$, and $T$ is an abelian subgroup which is diagonalizable. 
We will consider two separate cases: (a) $G=R$ is solvable, and (b) $G$ is semisimple, i.e., $\textrm{rad}(G)$ is trivial. 
We refer the reader to \cite{raghunathan} and to \cite{auslander} for general facts concerning lattices in 
Lie groups and the structure of solvmanifolds respectively.  

\subsection{Solvmanifolds} 
First suppose that $G$ is nilpotent and $M=G/H$ where $H$ is a closed subgroup.  Such a space is called a {\it nilmanifold.}  
Then $M$ is diffeomorphic to a product $\mathbb{R}^s\times M_0$ where $M_0$ is a {\it compact} smooth manifold of the form 
$N/\Gamma$ where $N$ is a connected nilpotent Lie group and $\Gamma$ is a discrete subgroup of $N$.  
By Theorem \ref{bh}, $M_0$ is parallelizable and so $M$ itself is parallelizable. 

Suppose that $G$ is a solvable Lie group.  In this case a homogeneous space $M=G/H$ is known as a solvmanifold.  
For basic facts about the structure of solvmanifolds, some of which will be recalled below, see \cite{auslander}.  
Auslander and Tolimieri showed that $M$ is diffeomorphic to 
the total space of a vector bundle over a compact solvmanifold, as conjectured by Mostow. 
Unlike in the case of nilmanifolds, solvmanifolds are not even stably parallelizable in general.  For example
the Klein bottle and the M\"obius band are solvmanifolds.   It turns out that any solvmanifold is an Eilenberg-MacLane 
space $K(\pi,1)$ and that if it is compact, then its diffeomorphism type is determined by its fundamental group. 
Thus the span of 
a solvmanifold is an invariant of its fundamental group.    The fundamental group $\Gamma$ of a compact solvmanifold 
 is {\it strongly polycyclic}, that is, there is a filtration 
\[ \Gamma=\Gamma_0>\Gamma_1>\cdots>\Gamma_n>\Gamma_{n+1}=1 \eqno{(18)}\]
where each $\Gamma_{i+1}$ is normal in $\Gamma_i$ and $\Gamma_i/\Gamma_{i+1}\cong \mathbb Z.$ 
Choose an element $\alpha_i\in \Gamma_i$ which maps to the generator of $\Gamma_i/\Gamma_{i+1}$. 
Then $\Gamma_i\cong \Gamma_{i+1}\rtimes \mathbb Z$ where the action of $\mathbb Z$ on $\Gamma_{i+1}$ 
is given by the restriction to $\Gamma_{i+1}$ of the conjugation by $\alpha_i$.  

While any finitely generated torsionless nilpotent group is a uniform lattice in a connected 
nilpotent Lie group, the analogous statement for solvable groups is false in general. 
For example, the fundamental group $\Gamma=\langle x,y\mid xyx^{-1}y\rangle=\mathbb Z\rtimes \mathbb Z$ of the 
Klein bottle cannot be a lattice in a connected Lie group $G$. (Otherwise the Klein bottle 
would be parallelizable.)  
This makes the vector field problem for (compact) solvmanifolds nontrivial and interesting. 

The following well-known result due Auslander and Szczarba \cite{auslander-szczarba} says that 
the structure group of the tangent bundle of a $d$-dimensional solvable manifold can be reduced to 
the diagonal subgroup of the orthogonal group $O(d)$.  Thus the manifold is close to being parallelizable. 

\begin{theorem} {\em (Auslander and Szczarba \cite{auslander-szczarba}.)}
Let $M$ be a compact solvmanifold of dimension $d$.  Then there exists line bundles $\xi_1,\ldots,\xi_d,$ 
such that $\tau M\cong \xi_1\oplus \cdots\oplus \xi_d$.  In particular all Pontrjagin classes of $M$ are trivial.
\end{theorem}

Using the fact that $M$ fibres over a circle, it can be seen easily that one of the line bundles, 
say $\xi_1$ may be taken to be trivial.
As an immediate corollary, one obtains the following 

\begin{theorem} 
Let $M$ be a compact solvmanifold of dimension $d$.  Then there exists a smooth covering $\wt M\to M$ 
of degree $2^k, k<d,$ such that $\wt{M} $ is parallelizable.  
\end{theorem}

Next we obtain a criterion for the (stable) parallelizability in a special case
and a criterion for the orientability in the general case. 

Let $A\in GL(n,\mathbb{Z})$ and let $\Gamma=\Gamma(A):=\mathbb{Z}^n\rtimes \mathbb{Z}$ be the extension of 
$\mathbb{Z}$ by $\mathbb{Z}^n$ where the $\mathbb{Z}$-action on $\mathbb{Z}^n$ is generated by $A$.   The smooth compact 
solvmanifold $M=M(A)$ with fundamental group $\Gamma$ may be described as follows: Let $\alpha:\mathbb T\to\mathbb T$ 
be the diffeomorphism of the $n$-dimensional torus $\mathbb T=\mathbb R^n/\mathbb Z^n$ defined by the linear automorphism of $\mathbb R^n$, $v\to Av$.  
The `mapping circle' $M=\mathbb T^n\times I/\!\!\sim$, where $(a,0)$ is identified to $(\alpha(a), 1)$, is a smooth manifold.  We have $T_{(a,x)}(\mathbb T\times I)
=T_a\mathbb T\times \mathbb{R}\cong \mathbb R^{n}\times \mathbb R$ for all $a\in \mathbb T, x\in I$.
The total space of the tangent bundle of $M$ has the 
following description: $TM=T(\mathbb T\times I)/\!\!\sim$ where $(a, 0;v, s)\in T_{(a, 0)} (\mathbb T\times I)$ is identified with $(\alpha(a), 1; Av, s)\in 
T_{(\alpha(a), 1)}(\mathbb T\times I)$ for $v\in T_a\mathbb T,  s\in T_0I=T_1I=\mathbb{R}$.  
The projection $\mathbb T\times I\to I$ induces a fibre bundle projection $\pi: M\to \mathbb S^1$ with fibre $\mathbb T$.  Hence we obtain 
an isomorphism $\tau M\cong \pi^*(\tau \mathbb S^1) \oplus \eta=\epsilon\oplus \eta$, where $\eta$ is the vertical bundle that 
restricts to the tangent bundle on the fibres of $\pi$.

\begin{theorem} \label{torusbundle} {\em (Sankaran, unpublished.)}
Let $A\in GL(n,\mathbb Z)$. With the above notation, 
the manifold $M=M(A)$ is parallelizable if $\det(A)=1$ and is not orientable---hence not stably parallelizable---if $\det(A)=-1$.
\end{theorem}
\begin{proof}  First assume that $\det(A)=1$.  We shall show that $\eta$ is trivial.  Let $\sigma:I\to GL(n,\mathbb R)$ be a smooth path such that $\sigma(0)=I_n,$ the identity 
matrix, and, $\sigma(1)=A$.  We will write $A_t$ to denote $\sigma(t)$.   We have $\tau \mathbb T=n\epsilon$ with total space 
$\mathbb T\times \mathbb R^n$. 
 The standard basis of $\mathbb R^n$ yields vector fields $X_1, \ldots, X_n$ on $\mathbb T\times I$ defined as follows: 
 $X_j(a, x)=(a,x; A_xe_j, 0)\in \mathbb T\times I\times \mathbb R^n\times \mathbb R$.  Note that $X_j(\alpha(a),1)=(\alpha(a), 1; Ae_j,  0)\sim 
 (a, 0; e_j, 0)=X_j(a,0)  $.  Hence $X_j$ descends to a well-defined smooth vector field, again denoted $X_j$ 
 on $M$.  Since $T\pi(X_j(p))=0$ for every $j$, it follows that $X_j, 1\le j\le n,$ are cross-sections of $\eta$.  
Since $A_t\in GL(n,\mathbb R)$ for all $t$, it is evident that $X_1,\ldots, X_n$ are everywhere linearly independent.
So $\eta$ is trivial, as was to be shown.  

Now suppose that $\det(A)=-1$.  Fix a point $a_0\in \mathbb T$ and choose a path $\sigma:I\to  \mathbb T$ from 
$a_0$ to $\alpha(a_0)$.  
Let $\theta:\mathbb S^1\to M$ be the embedding $\exp(2\pi it)\mapsto [\sigma(t), t]\in M, 0\le t\le 1$.  Consider the pull-back line 
bundle $\xi:=\theta^*(\Lambda^n(\eta))$ over $\mathbb S^1$.   We assert that $\xi$ is not orientable.  This readily implies 
that $\eta$ is non-orientable and hence $M$ itself is not orientable.  To prove the assertion, we need only observe that 
the total space of $\xi$ is obtained by identifying in $I\times \mathbb {R}$, the point $(0,t)\in I\times \mathbb R$ with 
$(1, (\det A) t)=(1,-t)$ for all $t\in \mathbb {R}$.  Thus $E(\xi)$ is homeomorphic to the M\"obius 
band and so $\xi$ is non-orientable.  
\end{proof} 

We remark that when $\det(A)=-1$, the double cover $\wt{M}$ of $M=M(A)$ corresponding to the 
subgroup $\mathbb Z^n\rtimes (2\mathbb Z)\subset \Gamma$ 
is just the group $\Gamma(A^2)$ and hence is parallelizable.  

One may generalize one part of the above theorem so as to obtain a criterion for the orientability of a solvmanifold.

Let $\Gamma$ be a strongly polycyclic group and write 
$\Gamma\cong \Gamma_1\rtimes \mathbb{Z}$ where $\Gamma_1$ is as in (18).  We let $A$ be the automorphism of $\Gamma_1$ that defines the action of $\mathbb Z$ on $\Gamma_1$.  Let $N:=M(\Gamma_1)$ be a compact solvmanifold 
with fundamental group $\Gamma_1$.  
Then $M=M(\Gamma)$, a solvmanifold with fundamental group $\Gamma$, may be obtained as the mapping circle of 
a diffeomorphism $\alpha:N\to N$ that induces $A$.   Now 
 $M$ fibres over 
the circle $\pi:M\to \mathbb S^1$ with fibre $N$ and we have a splitting $\tau M=\tau \mathbb S^1\oplus \eta$ where $\eta$ is the vertical bundle.   If $N$ is non-orientable, neither is $M$ since the normal bundle to the 
 fibre inclusion $N\hookrightarrow M$ is trivial.  
 
 Assume that $N$ is orientable.  If $\alpha: N\to N$ is orientable, 
 then $T\alpha: T_aN\to T_{\alpha(a)}N$ is orientation preserving.   As before, for any $a\in N$, we have $[a,0]=[a,1]$ in $M$ and 
 the tangent space $T_{[a,0]}M$ is obtained from $T(N\times I)=TN\times I\times \mathbb R$  
by identifying $(u,0;t)\in T_{(a,0)}(N\times I)$ with $(T_{a}\alpha( u), 1;t)\in T_{(\alpha(a),1)}(N\times I)$ where 
$u\in T_{a}N, t\in \mathbb {R}$.  Since $\tau M=\eta\oplus \epsilon$, the total space $E(\eta)\subset TM$ is the space of all 
vectors with vanishing last coordinate.   
Since $T_a\alpha:T_aN\to T_{\alpha(a)}N$ is orientation  preserving,  
we see that $\eta$ is orientable.  Hence $M$ is orientable.  On the other hand, if $\alpha$ is orientation reversing, 
then choosing a path $\sigma:I\to N$ from a point $a$ to $\alpha(a)$ we obtain an imbedding $\bar \sigma:\mathbb S^1\to M, 
\exp(2\pi it)\mapsto [\sigma(t), t], 0\le t\le 1$. 
The bundle $\Lambda^n(\eta)$ pulls-back via $\bar \sigma$ to a line bundle $\xi$ which is seen to be non-orientable. 
It follows that $\eta$ is non-orientable.  Hence $M$ is non-orientable.  Repeated application of this argument yields 
the following theorem.    
 
\begin{theorem}  \label{stronglypolygps} {\em (Sankaran, unpublished)}  Let $\Gamma$ be a strongly polycyclic group.  Let $\Gamma_i, 1\le i\le n+1,$ be as in (18).  Then $M(\Gamma)$ is orientable if each $M(\Gamma_i)$ is orientable and the action of $\Gamma_i/\Gamma_{i+1}\cong \mathbb Z$ on $M(\Gamma_i)$ is orientation preserving for $1\le i\le n$. 
\hfill $\Box$
\end{theorem}


\subsection{Homogeneous spaces for non-compact semisimple Lie groups}
Suppose that $G$ is a connected {\it non-compact} semisimple Lie group with finite centre.  
Let $K$ be a maximal compact subgroup of $G$.   
Then it is a consequence of Cartan (or Iwasawa) decomposition that $X:=G/K$ is diffeomorphic to $\mathbb{R}^n$ for some $n$.  
So the vector field problem for $G/K$ 
is uninteresting.  Suppose that $H$ is 
any connected compact subgroup of $G$.  Then $H$ is contained in a maximal compact connected subgroup $K$. 
One has a smooth fibre bundle $G/H\to G/K$ with fibre $K/H$. Since $G/K\cong \mathbb{R}^n$, the bundle is trivial and 
we have $G/H\cong \mathbb{R}^n\times K/H$.  Therefore $\Span(G/H)=\Span(\mathbb{R}^n\times K/H)
=\Span^0(K/H)+n=\Span^0(G/H)$ since $n\ge 1$.   

The manifold $X=G/K$ admits a $G$-invariant metric with respect to which it becomes a globally symmetric space.   
One may express $X$ as $\bar{G}/\bar K$ where $\bar G=G/Z(G)$ since the centre $Z(G)\subset K$. Note that $\bar G$ is 
a linear Lie group (via the adjoint representation). 
Thus we may assume, to begin with, 
that $G$ itself is linear.   Also, we will 
assume that $G$ has no (nontrivial) compact connected normal subgroup $N$.  (Any compact normal group is contained in $K$ 
and so $X\cong (G/N)/(K/N)$. So there is no loss of generality in such an assumption.)  

Let $\Gamma$ be a {\it uniform} lattice in $G$, that is, $\Gamma$ is a discrete 
group subgroup of $G$ such that  
$X_\Gamma:=\Gamma\backslash G/K$ is compact.  We will assume that $\Gamma$ is torsionless, that is, no 
element other than the identity has finite order. 
Then $X_\Gamma$ is a smooth manifold and quotient map $X\to X_\Gamma$ is a covering projection. 
The space $X_\Gamma$ is called a {\it locally symmetric space}.  Note that since $G$ is connected, given any element $g\in G$ 
the left translation by $g$ on $X$ is orientation preserving.  Hence $X_\Gamma$ is orientable.  

In a more general setting, one allows $\Gamma$ to be a (torsionless) lattice in the group $I(X)$ of {\it all} isometries of $X$.  In 
this generality, a locally symmetric space $X_\Gamma=I(X)/\Gamma$ is not necessarily orientable.  
The case considered above corresponds to the case where the lattice is contained in the identity component $G=I_0(G)$
 of $I(X)$.

Before proceeding further, we pause for an example.  

Let $G=SL(2, \mathbb{R})$, we have $X=\mathcal{H}$, 
the upper half space $\{z=x+iy\in \mathbb C\mid y>0\}$ (with the Poincar\'e metric).  If $\Gamma$ is a uniform torsionless lattice in 
$G$, then $X_\Gamma$ is 
a compact Riemann surface of genus $g\ge 2$.   By the uniformization theorem every 
compact Riemann surface arises in this manner.   Although $SL(2,\mathbb Z)$ is a lattice in $SL(2,\mathbb R)$, 
no finite index subgroup of it is uniform.  Explicit construction of a {\it uniform} lattice in $SL(2,\mathbb R)$ 
requires some preparation and will take us too far afield.

Borel \cite{borel-top} has shown that every (non-compact) semisimple Lie group $G$ admits both uniform and non-uniform lattices.  If $G$ is linear, 
then any lattice in $G$ has a finite index subgroup $\Gamma$ which is torsionless so that $\Gamma\backslash G/K$ 
is a smooth manifold.

The globally symmetric space $X$ has a compact dual $X_u:=U/K$ where $U$ is a maximal compact 
subgroup of the `complexification' of $G$, denoted $G_\mathbb{C}$, that contains $K$.  
The group $G_\mathbb{C}$ is characterized 
by the requirements that its Lie algebra is the complexification $\mathfrak{g}_\mathbb{C}=\mathfrak{g}\otimes_\mathbb{R}\mathbb{C}$ 
and $G\subset G_\mathbb{C}$. Such a group $G_\mathbb{C}$ exists in view of our assumption that $G$ is linear. 
When $G=SL(n,\mathbb{R}),$ we take $K=SO(n)\subset G$.  Then $G_\mathbb{C}=SL(n,\mathbb{C})$ and $U=SU(n)$, 
the special unitary group.  Hence the compact dual of $X$ is $X_u=SU(n)/SO(n)$.   We shall refer to 
$X_u$ also as the compact dual of a locally symmetric space $X_\Gamma$.

Returning to the general case of a compact locally symmetric space $X_\Gamma$, the well-known 
Hirzebruch proportionality principle says that the Pontrjagin numbers of $X_\Gamma$ are proportional 
to the corresponding Pontrjagin numbers of the compact dual $X_u$, the proportionality constant being 
dependent only on $X_\Gamma$.  Thus vanishing of the latter 
implies the vanishing of the former.  
See \cite{hirzebruch}.    What we need is a stronger version of the converse, namely: 
the non-vansihing of a characteristic {\it class} of $X_u$ implies the non-vanishing of the corresponding 
characteristic class of $X_\Gamma$.  This has been established by T. Kobayashi and K. Ono \cite{ko} in a more general 
setting.    The following 
theorem and its proof is essentially due to Lafont and Ray \cite{lafont-ray}, 
although they stated their result for 
characteristic {\it numbers}.   
The assertion concerning the Euler characteristic is 
well-known (cf. \cite{harder}).   

\begin{theorem}  \label{hpp} {\em  (\cite{ko}, \cite[Theorem A]{lafont-ray}.)}
With the above notation, if $\Gamma\subset G$ is a uniform lattice, then:
(i) $\chi(X_\Gamma)=c \chi(X_u)$ for some $c\ne 0$.\\ (ii)  If $\rank (K)=\rank (U)$, and if some Pontrjagin class $p_i(X_u)\ne 0$, then $p_i(X_\Gamma)\ne 0$.    
\end{theorem}  
\begin{proof}
(i). We may assume that $X_\Gamma$ is even-dimensional.    If  $\chi(X_u)=0$, then the {\it Euler class} $e(X_u)$ of $X_u$ vanishes. 
Hence by \cite{ko} it follows that $e(X_\Gamma)=0$ and so $\chi(X_\Gamma)=0$.  On the other hand, suppose 
that $\chi(X_u)\ne 0$. Then $\rank(U)=\rank(K)$.  In this case, a result of Okun \cite{okun} says that there exist a finite index subgroup 
$\Lambda \subset \Gamma$ and 
{\it tangential map} $f:X_\Lambda\to X_u$ of non-zero degree.  That is, $f^*(\tau X_u)
\cong \tau X_\Lambda$ and $f_*:H_n(X_u;\mathbb R)\to H_n(X_\Lambda;\mathbb R)$ is non-zero where $n=\dim X$.  
By the naturality of the Euler class, $f^*(e(X_u))=e(X_\Lambda)$; see \cite{ms}.  Now 
$\chi(X_\Lambda)=\langle e(X_\Gamma), \mu_{X_\Lambda}\rangle =\langle f^*(e(X_{u})), \mu_{X_{\Lambda}}\rangle
=\langle e(X_u),f_*(\mu_{X_\Lambda})\rangle =\langle e(X_u), \deg (f) \mu_{X_u}\rangle =\deg(f).\chi (X_u)\ne 0$.  This shows that $\chi(X_\Lambda)\ne 0$.  Since $\Lambda$ has finite index in $\Gamma$, we have a covering projection $X_\Lambda \to X_\Gamma$ and so 
$\chi(X_\Lambda)=|\Gamma/\Lambda| \chi(X_\Gamma)$.  Therefore $\chi(X_\Gamma)=\deg(f) \chi(X_u)/|\Gamma/\Lambda|$.  This proves (i). 

(ii).  Suppose that $p_i(X_u)\ne 0$.  
We proceed as in (i) above and use the same notations.    Note that the tangent bundle of $X_\Gamma$ pulls back under the covering 
projection to that of $X_\Lambda$.   By the naturality of Pontrjagin classes, it suffices to show that $p_i(X_\Lambda)\ne 0$.  Since $f:X_\Lambda \to X_u$ 
is tangential $f^*(p_i(X_u))=p_i(X_\Lambda)$.  Since $\deg(f)\ne 0$, the induced map in rational cohomology $f^*:H^*(X_u;\mathbb Q)\to H^*(X_\Lambda;\mathbb Q)$ is a monomorphism. Hence $p_i(X_\Lambda)\ne 0$.
\end{proof}

The group $G_\mathbb{C}$ is also the complexification of $U$.  In particular $U$ is semisimple.  
The rank of $U$ is also called the (complex) rank of $G$.   However $K\subset G$ is not necessarily semisimple. 
For example, when $G=SL(2,\mathbb R)$, $K=SO(2)$ is abelian.    When the centre of $K$ is not discrete and $G$ is 
simple, the homogeneous space $X=G/K$ has the structure of a Hermitian symmetric domain.  Also the compact 
dual $X_u=U/K$ has the structure of a complex projective variety.    When $G$ is semisimple it is an almost 
direct product $G=G_1\cdots G_k$ where each $G_i$ is a simple normal subgroup of $G$.  By our assumption, 
none of the $G_i$ is compact.  Any maximal compact subgroup $K$ is likewise an almost product 
$K=K_1\cdots K_k$ where $K_i\subset G_i$ is a maximal compact subgroup of $G_i$.  Moreover 
$X=G/K$ is diffeomorphic to the Cartesian product $X_1\times \cdots \times X_k$ where $X_i=G_i/K_i$.  The 
$X_i$ are called the irreducible factors of $X$.  Correspondingly, one has a factorization $X_u$ of the compact 
dual into a product $X_{1,u}\times \cdots \times X_{k,u}$ where $X_{i,u}=U_i/K_i$, $U_i$ being the maximal 
compact subgroup of $G_{i,\mathbb C}$ that contains $K_i$.

As an application of Theorem \ref{torussimple} we obtain the following result. 

\begin{theorem} \label{locallysymmetric} {\em (Sankaran, unpublished.)}
Let $\Gamma$ be a torsionless uniform lattice in a linear connected semisimple Lie group $G$.   With the above notations, 
(i) $\Span(X_\Gamma)>0$ if and only if $\rank(G)>\rank(K)$. \\
(ii) Suppose that an irreducible factor $X_i$ of $X$ is a Hermitian 
symmetric space where $G_i$ is not locally isomorphic to $SL(2,\mathbb R)$.  Then 
$X_\Gamma$ is not stably parallelizable.\\ 
(iii)  Suppose that each simple factor of $G$ is either a complex Lie group or is locally isomorphic to $SO_0(1,k)$.   
Then there exists a finite 
index subgroup $\Lambda\subset \Gamma$ such that $X_\Lambda$ is stably parallelizable.  Such 
an $X_\Lambda$ is parallelizable if and only if $\rank(G)>\rank(K)$.
\end{theorem}

\begin{proof} (i).  This is a direct consequence of Theorem \ref{hpp}(i), since $\chi(X_\Gamma)=0$ if and only if 
$\chi(U/K)=0$ if and only if $\rank(G)>\rank (K)$. 

(ii).  Consider the factor $X_{i,u}=U_i/K_i$.  The assumption that $G_i$ is not locally isomorphic to $SL(2,\mathbb R)$ 
implies that $K_i$ is neither semisimple nor a torus.  
Since $U_i$ is simple, by Theorem \ref{torussimple}, we see that $p_1(X_{i,u})\ne 0$.  It follows that $p_1(X_u)\ne 0$ since 
$X_{i,u}$ is a direct factor of $X_u$.  By Theorem \ref{hpp}(ii), it follows that $p_1(X_\Gamma)\ne 0$.   

(iii).  Let $H$ be a simple factor of $G$ with $L\subset H$ being a maximal compact subgroup.   Let $Y=H/L$ and $Y_u$
its compact dual.
When $H$ is a simple complex Lie group with maximal compact subgroup $L$, 
its complexification is the product $H\times H$ with maximal compact subgroup $L\times L$. Hence the compact dual of 
$Y$ is 
the homogeneous space $L\times L/L$ where $L$ is embedded diagonally.  Consequently $Y_u$ is diffeomorphic to the Lie 
group $L$ and hence is parallelizable.  When $H$ is locally isomorphic to $SO_0(1,k)$, the symmetric space $Y=H/L$ is the 
hyperbolic space $\mathcal H^k$ and its compact dual is the sphere $\mathbb S^k$.   Our hypothesis implies that $X_u$ 
is a product of spheres and a Lie group $M$ (possibly trivial). 
Thus $X_u$ is stably parallelizable.    
By Okun's theorem \cite{okun}, there 
exists a finite index subgroup $\Lambda\subset \Gamma$ such that 
there exists a tangential map $f: X_\Lambda\to X_u$; thus $f^*(\tau(X_u))\cong \tau X_\Lambda)$. 
It follows that $X_\Lambda$ is stably parallelizable and that it is parallelizable if and only if either one of the spheres 
is odd dimensional or $M$ is positive dimensional (see Theorems \ref{product-sphere}, \ref{product-span}).  The last 
condition is equivalent to $\chi(X_u)=0$, which is itself equivalent to the requirement that $\rank(G)>\rank(K)$.  
\end{proof}

\noindent
{\bf Acknowledgments.} I am grateful to Professor Peter Zvengrowski for sharing with me his insights into the vector 
field problem and for long years of collaboration. I am grateful to the referees for their very thorough reading of the 
 paper, for their comments, and for pointing out numerous errors.  One of them also pointed out that 
 Theorem \ref{product-sphere} and parts of Theorem \ref{product-span} were proved in the paper of 
 E. B. Staples \cite{staples}.   Also, I thank J\'ulius Korba\v{s}, Arghya Mondal, 
 Avijit Nath, and Peter Zvengrowski  for their comments and for pointing out errors.   I thank 
Mahender Singh for the invitation to participate in the Seventh East Asian Conference on Algebraic Topology 
held in December 2017 at IISER Mohali and for his interest in publishing these 
notes as part of the Conference proceedings.


\end{document}